\documentclass[reqno, 11pt]{amsart}
\usepackage{etex}
\usepackage[utf8]{inputenc}
\usepackage[english]{babel}
\usepackage{amssymb,amsmath,amsthm}
\usepackage{graphicx,geometry,caption,subcaption}
\usepackage{float}
\usepackage{faktor}
\usepackage{enumerate}
\usepackage{tikz}
\usetikzlibrary{patterns}
\usepackage{pgfplots}
\usepackage[all]{xy}
\usepackage{tikz-cd}
\usepackage{comment}
\usepackage{cite}
\usepackage{setspace}
\usepackage{amssymb,amstext}
\usepackage{xcolor}
\pgfplotsset{compat=1.10}
\usepackage{multirow}
\usepackage{xcolor}
\usepackage{ifthen}

\numberwithin{equation}{section}
\newcommand{\extp}{\@ifnextchar^\@extp{\@extp^{\,}}}
\def\extp^#1{\mathop{\bigwedge\nolimits^{\!#1}}}
\makeatother

\theoremstyle{plain}
\newtheorem{theorem}{Theorem}[section]
\newtheorem{proposition}{Proposition}[section]
\newtheorem{corollary}{Corollary}[section]
\newtheorem{lemma}{Lemma}[section]

\newtheorem{remark}{Remark}

\theoremstyle{definition}
\newtheorem{defi}{Definition}[section]
\newtheorem{rem}{Remark}[section]

\DeclareMathOperator{\im}{Im}
\DeclareMathOperator{\Sp}{Sp}
\DeclareMathOperator{\Tr}{Tr}
\DeclareMathOperator{\Hom}{Hom}
\newcommand{\Ytree}[3]{
	\begin{tikzpicture}[baseline=-3pt, scale=.6]
		
		\draw[thick] (0,0) -- (0,1);
		\draw[thick] (0.7,-0.7) -- (0, 0);
		\draw[thick] (-0.7,-0.7) -- (0,0);
		
		\node[above] at (0,1) {$#1$};
		\node[below] at (0.9,-0.5) {$#2$};
		\node[below] at (-0.9,-0.6) {$#3$};
		
	\end{tikzpicture}
}

\newcommand{\tree}[4]{
\begin{tikzpicture}[baseline=-3pt, scale=.6]

\draw[thick] (-1,0) -- (1,0);
\draw[thick] (-1,0.3) -- (-1, -0.3);
\draw[thick] (1,0.3) -- (1, -0.3);

\node[above] at (-1,0.3) {$#1$};
\node[below] at (-1,-0.3) {$#2$};
\node[above] at (1,0.3) {$#4$};
\node[below] at (1,-0.3) {$#3$};

\end{tikzpicture}
}

\begin{document}

\title{Finite type invariants in low degrees and the Johnson filtration}

% author information
\author{Wolfgang Pitsch}
\author{Ricard Riba}
\address{Universitat Autònoma de Barcelona, Departament de Matemàtiques, Bellaterra, Spain}
\email{Wolfgang.Pitsch@uab.cat}
\address{Universitat Politècnica de Catalunya, Departament de Matemàtiques, Manresa, Spain}
\email{ricard.riba.garcia@upc.edu}
\thanks{This work was supported by MEC grantPID2020-116481GB-I00
	and the AGUR grant 2021-SGR-01015}



\date{\today}

%The Rohlin invariant and the algebraic construction of invariants of homology spheres

\begin{abstract}
We study the behaviour of the Casson invariant $\lambda$, its square, and Othsuki's second invariant $\lambda_2$ as functions on the Johnson subgroup of the mapping class group. We show that since $\lambda$ and $d_2 := \lambda_2 - 18 \lambda^2$ are invariants that induce morphisms on respectively the second and the third level of the Johnson filtration, they never vanish on any level of this filtration. In contrast we prove  that the invariant $\lambda_2-18\lambda^2 +3\lambda$ vanishes on the fifth level of the Johnson filtration, $\mathcal{M}_{g,1}(5)$. As a consequence, we prove that the Poincaré sphere is an instance of homology 3-sphere that does not admit any Heegaard splitting with gluing map an element in $\mathcal{M}_{g,1}(5)$. Finally we establish a surgery formula for Othsuki's second invariant $\lambda_2$.
\end{abstract}

\maketitle

%%%%%%%%%%%%%%%%%%%%%%%%%%%%%%%%%%%%%%%%%%%%%%%%%%%%%%%%%%%%
%%%%%%%%%%%%%%%%%%%%%%%%%%%%%%%%%%%%%%%%%%%%%%%%%%%%%%%%%%%%
\section{Introduction}\label{sec:intro}
%%%%%%%%%%%%%%%%%%%%%%%%%%%%%%%%%%%%%%%%%%%%%%%%%%%%%%%%%%%%
%%%%%%%%%%%%%%%%%%%%%%%%%%%%%%%%%%%%%%%%%%%%%%%%%%%%%%%%%%%%

Let $\Sigma_{g,1}$ denote an oriented surface of genus $g$ with a marked $2$-disk and denote by $\mathcal{M}_{g,1}$ its mapping class group relative to the marked disk. From the theory of Heegaard splittings we learn that any oriented $3$-manifold up to diffeomorphism can be obtained for some $g \geq 1$ by cutting the oriented $3$-sphere $\mathbb{S}^3$ along a standardly embedded copy of $\Sigma_{g,1}$ and gluing back the two resulting handlebodies using some element in $\mathcal{M}_{g,1}$. If one considers not any element in the mapping class group but only elements in the Torelli subgroup $\mathcal{T}_{g,1} \subseteq \mathcal{M}_{g,1}$, i.e. the kernel of the action on $H_1(\Sigma_{g,1},\mathbb{Z})$, then homologically it is like gluing back by the identity and hence the resulting manifold is an integral homology $3$-sphere. Furthermore, by a result of Morita~\cite{mor}, any integral homology sphere can be obtained from an element in $\mathcal{T}_{g,1}$ for some $g \geq 1$. The Torelli group is the first step in a descending filtration of the mapping class group known as the Johnson filtration
\[
\dots \mathcal{M}_{g,1}(k) \subseteq \cdots \subseteq \mathcal{M}_{g,1}(2) \subseteq \mathcal{M}_{g,1}(1)= \mathcal{T}_{g,1} \subseteq \mathcal{M}_{g,1}.
\]
One can then ask which $3$-manifolds do the different steps in this filtration parametrize. In a series of papers, Morita~ \cite{mor}, Pitsch \cite{pitsch3} and Faes \cite{faes2} showed that the groups $\mathcal{M}_{g,1}(k)$ for $k=2,3$ and $4$ respectively parametrize all integral homology spheres, and one could then wonder if this phenomenon continues, i.e. if any homology sphere can be built via a Heegaard splitting from an element arbitrary low in the Johnson filtration for an appropriate value of the genus $g$. We will show in this work that this is not the case: not all integral homology spheres can be built from elements in $\mathcal{M}_{g,1}(5)$.

To do so we study the behaviour of the first finite type invariants in the sense of Ohtsuki~\cite{ohtsuki1} when viewed as functions on the Johnson subgroups $\mathcal{M}_{g,1}(2) \subseteq \mathcal{T}_{g,1}$. Finite type invariants are defined from the filtration of the group algebra $\mathbb{Q}\mathcal{T}_{g,1}$ by the powers of the augmentation ideal, and they are therefore related to the lower central series of the Torelli group.  To understand how they interact with the Johnson filtration,  we consider  for an invariant say $F$,  the functions $F_g$ it defines   on the Johnson subgroup  $\mathcal{M}_{g,1}(2)$ and to these we associate the trivialized $2$-cocycles $C_g^F(\phi,\psi)=F_g(\phi) + F_g(\psi)- F_g(\phi\psi)$ which measure the failure of $F_g$ to be a homomorphism. As a first result we have (Theorem~\ref{thm:CassonHomomorphInv} and Lemma~\ref{lem:type2givesbilincocycle}):

%%%%%%%%%%%%%%%%%%%%%%%%%%%%%%%%%%%%%%%%%%%%%%%%%%%%%%%%%
\begin{theorem}\label{thm:intoinvdegmoins2}
%%%%%%%%%%%%%%%%%%%%%%%%%%%%%%%%%%%%%%%%%%%%%%%%%%%%%%%%%
\begin{enumerate}
	\item  Up to a multiplicative constant the Casson invariant is the unique invariant that induces a homomorphism on the Johnson subgroups $\mathcal{M}_{g,1}(2)$.
	\item An invariant of degree $\leq 2$ has a corresponding $2$-cocycle that is a bilinear form on the Johnson subgroups $\mathcal{M}_{g,1}(2)$.
\end{enumerate}
\end{theorem}
%%%%%%%%%%%%%%%%%%%%%%%%%%%%%%%%%%%%%%%%%%%%%%%%%%%%%%%%%

The natural domain for this bilinear form is the group $H_1(\mathcal{M}_{g,1}(2);\mathbb{Q})$ which has been computed by Massuyeau-Faes~\cite{MF}, and is isomorphic to $\mathbb{Q} \oplus \mathcal{M}_{g,1}(3)/\mathcal{M}_{g,1}(4) \otimes  \mathbb{Q} \oplus \mathcal{M}_{g,1}(2)/\mathcal{M}_{g,1}(3) \otimes  \mathbb{Q}$; further analysis shows (cf. Theorem~\ref{thm:bilinearonKsplits}):

%%%%%%%%%%%%%%%%%%%%%%%%%%%%%%%%%%%%%%%%%%%%%%%%%%%%%%%%%
\begin{theorem}
%%%%%%%%%%%%%%%%%%%%%%%%%%%%%%%%%%%%%%%%%%%%%%%%%%%%%%%%%
The $2$-cocycle of an invariant of degree $\leq 2$ splits as the sum of a bilinear form on $\mathbb{Q}$ and a bilinear form on $\mathcal{M}_{g,1}(2)/\mathcal{M}_{g,1}(3) \otimes  \mathbb{Q}$.
\end{theorem}
%%%%%%%%%%%%%%%%%%%%%%%%%%%%%%%%%%%%%%%%%%%%%%%%%%%%%%%%%

The $\mathbb{Q}$-vector space of normalized (i.e. that are $0$ on the standard sphere) finite type invariants of degree $ \leq 2$ is of dimension $3$, generated by the Casson invariant $\lambda$, its square $\lambda^2$ and the second Ohtsuki invariant $\lambda_2$. A direct computation shows that the bilinear form on $\mathbb{Q}$ corresponds to the square of the Casson invariant  $\lambda^2$, and applying the previous result to $\lambda_2$ we show that there is an invariant $d_2$ of type $2$ which is a group homomorphism on the subgroups $\mathcal{M}_{g,1}(3)$.

The fact that both the Casson invariant and our invariant $d_2$ are group homomorphisms on one of the first $4$ steps of the Johnson filtration is  the key to prove the first geometrical application of our results (cf. Theorem~\ref{thm:invrastJohnsfilt}).

%%%%%%%%%%%%%%%%%%%%%%%%%%%%%%%%%%%%%%%%%%%%%%%%%%%%%%%%%
\begin{theorem}
%%%%%%%%%%%%%%%%%%%%%%%%%%%%%%%%%%%%%%%%%%%%%%%%%%%%%%%%%
For $g$ large enough, neither the Casson invariant nor the invariant $d_2$ vanish on any term of the Johnson filtration. 
\end{theorem}
%%%%%%%%%%%%%%%%%%%%%%%%%%%%%%%%%%%%%%%%%%%%%%%%%%%%%%%%%
For the Casson invariant this result is due to Hain, we nevertheless include it here as our proof for both invariants is parallel. From this we can finally prove our main geometrical result (cf. Theorem~\ref{thm:filtisnotfull}):

%%%%%%%%%%%%%%%%%%%%%%%%%%%%%%%%%%%%%%%%%%%%%%%%%%%%%%%%%
\begin{theorem}
%%%%%%%%%%%%%%%%%%%%%%%%%%%%%%%%%%%%%%%%%%%%%%%%%%%%%%%%%
The invariant $d_2 +3\lambda=\lambda_2 + 3 \lambda -18\lambda^2$ vanishes on $\mathcal{M}_{g,1}(5)$. Its value on the Poincaré sphere $\mathbb{P}$ is $24$, in particular $\mathbb{P}$ can not be described as a Heegaard splitting with gluing map an element in some of the groups $\mathcal{M}_{g,1}(5)$. 
\end{theorem}
%%%%%%%%%%%%%%%%%%%%%%%%%%%%%%%%%%%%%%%%%%%%%%%%%%%%%%%%%
	
	Finally, in Section~\ref{sec:billambda2} we compute the bilinear form associated to Ohtsuki's invariant $\lambda_2$. This gives an algebraic version of an explicit surgery formula for this invariant.
	
{\bf Acknowledgements: } It is a pleasure to thank G.~Massuyeau and Q.~Faes for their remarks on a first version of this work. In particular they pointed out a boundary problem in our original proof of Proposition~\ref{prop:ab_M(2)/T(5)}: the core computation that we use in our argument is due to them. We also thank the anonymous referee for valuable suggestions and corrections that improved the exposition.

%%%%%%%%%%%%%%%%%%%%%%%%%%%%%%%%%%%%%%%%%%%%%%%%%%%%%%%%%%%%
%%%%%%%%%%%%%%%%%%%%%%%%%%%%%%%%%%%%%%%%%%%%%%%%%%%%%%%%%%%%
\section{Preliminary results}\label{sec:Prel}
%%%%%%%%%%%%%%%%%%%%%%%%%%%%%%%%%%%%%%%%%%%%%%%%%%%%%%%%%%%%
%%%%%%%%%%%%%%%%%%%%%%%%%%%%%%%%%%%%%%%%%%%%%%%%%%%%%%%%%%%%

%%%%%%%%%%%%%%%%%%%%%%%%%%%%%%%%%%%%%%%%%%%%%%%%%%%%%%%%%%%%
\subsection{Heegard splittings of integral homology spheres}\label{subsec:Heegsplitt}
%%%%%%%%%%%%%%%%%%%%%%%%%%%%%%%%%%%%%%%%%%%%%%%%%%%%%%%%%%%%

Let $\Sigma_{g,1}$ be an oriented surface of genus $g$ standardly embedded in the oriented $3$-sphere $\mathbb{S}^3$ with a marked small disc. 

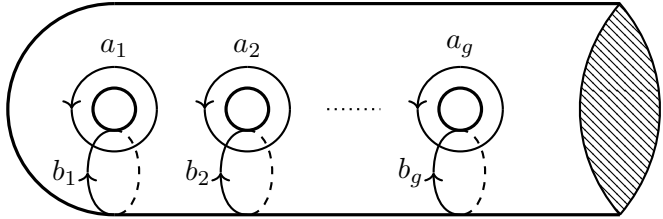
\begin{figure}[H]
	\begin{center}
		\begin{tikzpicture}[scale=.7]
			\draw[very thick] (-4.5,-2) -- (5,-2);
			\draw[very thick] (-4.5,2) -- (5,2);
			\draw[very thick] (-4.5,2) arc [radius=2, start angle=90, end angle=270];
			
			\draw[very thick] (-4.5,0) circle [radius=.4];
			\draw[very thick] (-2,0) circle [radius=.4];
			\draw[very thick] (2,0) circle [radius=.4];

			\draw[thick, dotted] (-0.5,0) -- (0.5,0);

			\draw[<-,thick] (1.2,0) to [out=90,in=180] (2,0.8);
			\draw[thick] (2.8,0) to [out=90,in=0] (2,0.8);
			\draw[thick] (1.2,0) to [out=-90,in=180] (2,-0.8) to [out=0,in=-90] (2.8,0);
			
			\draw[<-, thick] (-5.3,0) to [out=90,in=180] (-4.5,0.8);
			\draw[thick] (-3.7,0) to [out=90,in=0] (-4.5,0.8);
			\draw[thick] (-5.3,0) to [out=-90,in=180] (-4.5,-0.8) to [out=0,in=-90] (-3.7,0);
			
			\draw[<-,thick] (-2.8,0) to [out=90,in=180] (-2,0.8);
			\draw[thick] (-1.2,0) to [out=90,in=0] (-2,0.8);
			\draw[thick] (-2.8,0) to [out=-90,in=180] (-2,-0.8) to [out=0,in=-90] (-1.2,0);

			\draw[thick] (-4.5,-0.4) to [out=180,in=90] (-5,-1.2);
			\draw[->,thick] (-4.5,-2) to [out=180,in=-90] (-5,-1.2);
			\draw[thick, dashed] (-4.5,-0.4) to [out=0,in=0] (-4.5,-2);
			\draw[thick] (-2,-0.4) to [out=180,in=90] (-2.5,-1.2);
			\draw[->,thick] (-2,-2) to [out=180,in=-90] (-2.5,-1.2);
			\draw[thick, dashed] (-2,-0.4) to [out=0,in=0] (-2,-2);
			\draw[thick] (2,-0.4) to [out=180,in=90] (1.5,-1.2);
			\draw[->,thick] (2,-2) to [out=180,in=-90] (1.5,-1.2);
			\draw[thick, dashed] (2,-0.4) to [out=0,in=0] (2,-2);
			
			\node [left] at (-5,-1.2) {$b_1$};
			\node [above] at (-4.5,0.8) {$a_1$};
			\node [left] at (-2.5,-1.2) {$b_2$};
			\node [left] at (1.5,-1.2) {$b_g$};
			\node [above] at (-2,0.8) {$a_2$};
			\node [above] at (2,0.8) {$a_g$};
			
			\draw[thick,pattern=north west lines] (5,-2) to [out=130,in=-130] (5,2) to [out=-50,in=50] (5,-2);
			
		\end{tikzpicture}
	\end{center}
	\caption{Homology basis of $H_1(\Sigma_{g,1};\mathbb{Z})$}
	\label{fig:homology_basis}
\end{figure}

We fix on $\Sigma_{g,1}$ the $2g$ curves $a_i,b_i$ shown in Figure~\ref{fig:homology_basis}, whose homology classes give a symplectic basis of $H= H_1(\Sigma_{g,1},\mathbb{Z})$ endowed with its canonical intersection form $\omega:H \times H \rightarrow \mathbb{Z}$. The surface $\Sigma_{g,1}$ separates $\mathbb{S}^3$ into two genus $g$ handlebodies $\mathbb{S}^3=\mathcal{H}_g\cup -\mathcal{H}_g$ with opposite induced orientation.
Notice that by construction, the homology classes $b_i$, $1 \leq i \leq g$ give a basis of the Lagrangian $B = \ker (H_1(\Sigma_{g,1},\mathbb{Z}) \rightarrow H_1(\mathcal{H}_{g},\mathbb{Z}) )$ and the homology classes $a_i$, $1 \leq i \leq g$ give a basis of the supplementary  Lagrangian $A = \ker (H_1(\Sigma_{g,1},\mathbb{Z}) \rightarrow H_1(-\mathcal{H}_{g},\mathbb{Z}) ).$

Denote by $\mathcal{M}_{g,1}$ the mapping class group of $\Sigma_{g,1},$ i.e. the group of orientation-preserving diffeomorphisms of $\Sigma_{g,1}$ that are the identity on our fixed disc modulo isotopies which again fix that small disc point-wise. The connected sum along the marked disk of $\Sigma_{g,1}$ with a marked torus, slightly abusing notation, gives an embedding $\Sigma_{g,1} \rightarrow \Sigma_{g+1,1}$. By a classical result of Nielsen \cite{Nil} (see \cite[Theorem 2.7.I]{Iva} for a modern reference), extending a diffeomorphism on $\Sigma_{g,1}$ that fixes the marking point-wise by the identity over the torus, we get an injective morphism $\mathcal{M}_{g,1} \hookrightarrow \mathcal{M}_{g+1,1}$ called the stabilization map. The action of the mapping class group on the $H$ gives us a surjective map $\mathcal{M}_{g,1} \rightarrow \Sp_{2g}(\mathbb{Z})$, whose kernel is by definition the Torelli group $\mathcal{T}_{g,1}$.

The embedding of $\Sigma_{g,1}$ in $\mathbb{S}^3$ determines three natural subgroups of $\mathcal{M}_{g,1}$: the subgroup $\mathcal{B}_{g,1}$ of mapping classes that extend to the inner handlebody
$\mathcal{H}_g$, the subgroup $\mathcal{A}_{g,1}$ of mapping classes that extend to the outer handlebody $-\mathcal{H}_g$ and their intersection $\mathcal{AB}_{g,1}$, which are the mapping classes that extend to $\mathbb{S}^3$. Denote by $\mathcal{V}$ the set of diffeomorphism classes of  closed  oriented smooth $3$-manifolds.
By the theory of Heegaard splittings we know that any element in $\mathcal{V}$ can be obtained by cutting $\mathbb{S}^3$ along $\Sigma_{g,1}$ for some $g$ and gluing back the two handlebodies by some element of the mapping class group $\mathcal{M}_{g,1}.$
We write $\mathcal{H}_g \cup_{\iota_g} -\mathcal{H}_g$ for our preferred Heegaard splitting of the standard sphere, so that if $\phi \in \mathcal{M}_{g,1}$ is the mapping class used to glue back the handlebodies, from $\partial \mathcal{H}_g$ to $\partial (-\mathcal{H}_g)$, the resulting manifold is represented by $\mathbb{S}^3_\phi = \mathcal{H}_g \cup_{\iota_g \phi} -\mathcal{H}_g$. 
The lack of injectivity of this construction is controlled by the handlebody subgroups $\mathcal{A}_{g,1}$ and $\mathcal{B}_{g,1}$; J.~Singer~\cite{singer} proved that there is a bijection:
\begin{align*}
	\lim_{g\to \infty}\mathcal{A}_{g,1}\backslash\mathcal{M}_{g,1}/\mathcal{B}_{g,1} & \longrightarrow \mathcal{V}, \\
	\phi & \longmapsto \mathbb{S}^3_\phi=\mathcal{H}_g \cup_{\iota_g\phi} -\mathcal{H}_g.
\end{align*}

If we restrict this map to mapping classes that contain an element of the Torelli group, denoting by $\mathcal{S}$ the set of oriented diffeomorphism classes of integral homology $3$-spheres, we have a bijection~\cite{mor}:
\begin{align*}
	\lim_{g\to \infty}\mathcal{A}_{g,1}\backslash\mathcal{T}_{g,1}/\mathcal{B}_{g,1} & \longrightarrow \mathcal{S}, \\
	\phi & \longmapsto \mathbb{S}^3_\phi=\mathcal{H}_g \cup_{\iota_g\phi} -\mathcal{H}_g.
\end{align*}
In both cases the limit is taken over the stabilization maps. 

In the sequel we will need a more precise description of the homology action of $\mathcal{AB}_{g,1}$.  Writing the symplectic matrices according to the decomposition $H = A\oplus B$, the action of the subgroup $\mathcal{AB}_{g,1}$ on $H$ has precisely the same image as the diagonal embedding:
\[
\begin{array}{rcl}
	GL_g(\mathbb{Z}) & \longrightarrow & \Sp_{2g}(\mathbb{Z}) \\
	G & \longmapsto & \left( \begin{matrix}
							G& 0 \\
							0 & {}^tG^{-1}
						\end{matrix}
		 \right).
\end{array}
\]
Notice that the action on either $A$ or $B$ fully determines the homological action of an element in $\mathcal{AB}_{g,1}$.
For further details, we refer the reader to \cite{suzuki} and \cite[Sec. 1.2]{pitsch}.

%%%%%%%%%%%%%%%%%%%%%%%%%%%%%%%%%%%%%%%%%%%%%%%%%%%%%%%%%%%%
\subsection{Finite type invariants}\label{subsec:finitetypeinv}
%%%%%%%%%%%%%%%%%%%%%%%%%%%%%%%%%%%%%%%%%%%%%%%%%%%%%%%%%%%%

In  \cite{ohtsuki1} Ohtsuki introduced the notion of an invariant of finite type, related to a decreasing filtration of the vector space $\mathbb{Q}\mathcal{S}$ generated by $\mathcal{S}$. Very concisely, given an invariant $F: \mathcal{S} \rightarrow \mathbb{Q}$, we extend F linearly to $F: \mathbb{Q}\mathcal{S} \rightarrow \mathbb{Q}$. We call $F$ of \emph{finite type} if and only if it vanishes on one term of the filtration.
A number of natural choices for the filtration have appeared, by algebraic split links (the original one), borromean surgeries, Torelli surgeries, etc.   They were compared and showed to be essentially equal by Garoufalidis-Levine in  \cite{levine1}. We will hence choose the one best suited to our point of view, which is defined in terms of the lower central series of the Torelli group.

Composing  the Heegaard  splitting bijection due to Morita,
\[
\lim_{g\to \infty}\mathcal{A}_{g,1}\backslash\mathcal{T}_{g,1}/\mathcal{B}_{g,1} \stackrel{\sim}{\longrightarrow} \mathcal{S},
\]
with the canonical maps from the Torelli group gives for each $g \geq 1$ a map
\begin{equation}
\label{eq:Torelli-ZHS}
\mathcal{T}_{g,1} \longrightarrow \mathcal{S}
\end{equation}
that extends linearly to a map from the group algebra of the Torelli group
\[
\mathbb{Q}\mathcal{T}_{g,1} \longrightarrow \mathbb{Q}\mathcal{S}.
\]
%----------------------------------------------------------------------

Let $I\mathcal{T}_{g,1}$ denote the augmentation ideal, and for $k \geq 1$ denote by  $\mathcal{S}_k^{\mathcal{T}}$ the span of $\bigcup_g (I\mathcal{T}_{g,1})^k$ in $\mathbb{Q}\mathcal{S}$. This gives a decreasing filtration
\[
\mathbb{Q}\mathcal{S} = \mathcal{S}_0 \supseteq \mathcal{S}_1^{\mathcal{T}} \supseteq \mathcal{S}_2^{\mathcal{T}} \supseteq \cdots \mathcal{S}_k^{\mathcal{T}} \supseteq \cdots
\]
For instance $\mathcal{S}_1^\mathcal{T}$ is generated by the differences $\mathbb{S}^3 - \mathbb{S}^3_f$ for  $f \in \mathcal{T}_{g,1}$, and $\mathcal{S}_2^\mathcal{T}$ is generated by the formal sums $\mathbb{S}^3 -\mathbb{S}^3_f + \mathbb{S}^3_{fh} - \mathbb{S}^3_{h}$, for $f,h \in \mathcal{T}_{g,1}$. It turns out that this filtration has some  trivial steps, by \cite[Corollary 1.21 and Theorem 4]{levine3}, for all $k \geq 1$, $\mathcal{S}^\mathcal{T}_{2k} = \mathcal{S}^\mathcal{T}_{2k-1}
$. This incites to re-number the filtration by setting 
\[
\forall k \geq 1, \ \mathcal{S}_k= \mathcal{S}^\mathcal{T}_{2k} = \mathcal{S}^\mathcal{T}_{2k-1}.
\]

%%%%%%%%%%%%%%%%%%%%%%%%%%%%%%%%%%%%%%%%%%%%%%%%%%%%%%%%
\begin{defi}\label{def:finitetypeinv}
%%%%%%%%%%%%%%%%%%%%%%%%%%%%%%%%%%%%%%%%%%%%%%%%%%%%%%%%
Let $k \geq 0$ be an integer. An invariant $F :\mathcal{S} \rightarrow \mathbb{Q}$ is said to be of finite type $\leq k$ if and only if its linear extension to $\mathcal{S}_0= \mathbb{Q}\mathcal{S}$ vanishes on $\mathcal{S}_{k+1}$.
\end{defi}  
%%%%%%%%%%%%%%%%%%%%%%%%%%%%%%%%%%%%%%%%%%%%%%%%%%%%%%%

So an invariant is of finite type $0$ if and only if it is constant on $\mathcal{S}$. We will require all our invariants to be \emph{normalized}, i.e. that $F(\mathbb{S}^3)=0$, this means that a degree $0$ invariant is equal to $0$. Because $\mathcal{S}_1= \mathcal{S}^\mathcal{T}_{2} = \mathcal{S}^\mathcal{T}_{1}$, there are no invariants that are homomorphisms when restricted to the Torelli groups, a fact that can be proved directly \cite{pitsch} (the proof given there for integral-valued invariants applies equally well to rational invariants).
The first example of a non-trivial invariant of finite type is the Casson invariant $\lambda$, which is of degree $1$ and it is the first of a series of invariants $6\lambda=\lambda_1, \dots, \lambda_k,\dots$ constructed by Ohtsuki~\cite{ohtsuki1}.

It is a non-trivial fact due to Ohtsuki~\cite{ohtsuki1} that the vector spaces $\mathcal{S}_k/\mathcal{S}_{k+1}$ are all finite dimensional. In low degrees we have
\begin{enumerate}
	\item $\mathcal{S}_0/\mathcal{S}_1 = \mathbb{Q}$,detected by constant-valued invariants.
	\item $\mathcal{S}_1/\mathcal{S}_2 = \mathbb{Q}$, detected by $\lambda$, the Casson invariant.
	\item $\mathcal{S}_2/\mathcal{S}_3 = \mathbb{Q}^2$, detected by $\lambda^2$, the square of the Casson invariant,  and $\lambda_2$ the second Ohtsuki invariant.
\end{enumerate}

We introduce two descending series on the Torelli groups.  Firstly we have the lower central series, which for an arbitrary group $G$ is inductively defined by $\Gamma_1G=G$ and $\Gamma_{k+1}G=[G,\Gamma_k G]$ for $k\geq 1$. For the Torelli group we denote
$\mathcal{T}_{g,1}(k)=\Gamma_k \mathcal{T}_{g,1}$. It is also convenient to consider the radical closure of the lower central series:
\[
\sqrt{\mathcal{T}_{g,1}(k)} = \{ x \in \mathcal{T}_{g,1} \ | \ \exists n \geq 1 \textrm{ such that } x^n \in \mathcal{T}_{g,1}(k) \}
\]

This radical closure is closely related to the filtration by the augmentation ideals of the groups ring, since (see for instance \cite[Theorem IV.1.5]{Passi}) for any $n \geq 1$ and any $g \in \mathcal{T}_{g,1}$
\[
g-1 \in (I\mathcal{T}_{g,1})^n \Leftrightarrow g \in \sqrt{\mathcal{T}_{g,1}(n)}.
\]

As a consequence, taking into account the shift in grades between $\mathcal{S}_k$ and the lower central series,
if for $g\geq 0$ we denote $F_g:\mathcal{T}_{g,1}\rightarrow \mathbb{Q}$ the map resulting from precomposing a normalized invariant $F: \mathcal{S} \rightarrow \mathbb{Q}$ with the map $\mathcal{T}_{g,1} \rightarrow \mathcal{S}$ given in \eqref{eq:Torelli-ZHS}, we have:

%%%%%%%%%%%%%%%%%%%%%%%%%%%%%%%%%%%%%%%%%%%%%%%%%%%%%%%%%%%%%%%%%%%%%%
\begin{lemma}\label{lem:anulradsbgrpo}
%%%%%%%%%%%%%%%%%%%%%%%%%%%%%%%%%%%%%%%%%%%%%%%%%%%%%%%%%%%%%%%%%%%%%%
Let $F: \mathcal{S} \rightarrow \mathbb{Q}$ be a normalized invariant. If $F$ is of finite type $k$ then, for all $g \geq 0$, the function $F_g$ vanishes on $\sqrt{\mathcal{T}_{g,1}(2k+1)}$. 
\end{lemma}
%%%%%%%%%%%%%%%%%%%%%%%%%%%%%%%%%%%%%%%%%%%%%%%%%%%%%%%%%%%%%%%%%%%%%%

In \cite{lev} Levine studied the relation between finite type invariants of homology spheres and the structure of the Torelli group, particularly its lower central series.
Following his arguments, one can generalize \cite[Lemma 5.5]{lev} to the radical closure of the lower central series of the Torelli group as follows:

%%%%%%%%%%%%%%%%%%%%%%%%%%%%%%%%%%%%%%%%%%%%%%%%%%%%%%%%%%%%%%%%%%%%%%%%%%
\begin{proposition}
\label{prop:coc-finite-type}
%%%%%%%%%%%%%%%%%%%%%%%%%%%%%%%%%%%%%%%%%%%%%%%%%%%%%%%%%%%%%%%%%%%%%%%%%%
	Let $F$ be a normalized invariant of homology spheres of finite type $k$. Then for any $\varphi\in \sqrt{\mathcal{T}_{g,1}(k_1)}$, $\psi\in \sqrt{\mathcal{T}_{g,1}(k_2)}$ with $k_1+k_2>2k$, the following equality holds
	$$F(\mathbb{S}^3_{h_1 h_2})=F(\mathbb{S}^3_{h_1})+F(\mathbb{S}^3_{h_2}).$$
\end{proposition}
%%%%%%%%%%%%%%%%%%%%%%%%%%%%%%%%%%%%%%%%%%%%%%%%%%%%%%%%%%%%%%%%%%%%%%%%%%
\begin{proof}
	View $F$ as a function of the groups ring $\mathbb{Q}\mathcal{T}_{g,1}$. Observe that if $h_i\in \sqrt{\mathcal{T}_{g,1}(k_i)}$, then $h_i-1\in(I\mathcal{T}_{g,1})^{k_i}$ and so $(h_1-1)(h_2-1)\in (I\mathcal{T}_{g,1})^{k_1+k_2}\subseteq (I\mathcal{T}_{g,1})^{2k+1}$.
	Therefore,
	$$F(\mathbb{S}^3_{h_1 h_2}-\mathbb{S}^3_{h_1} -\mathbb{S}^3_{h_2})=F(h_1 h_2-h_1 -h_2-1)=F((h_1-1)(h_2-1))=0.$$
\end{proof}

As a direct consequence, we also get a radical closure analogue of \cite[Corollary 7.3]{MSS}:

%%%%%%%%%%%%%%%%%%%%%%%%%%%%%%%%%%%%%%%%%%%%%%%%%%%%%%%%%%%%%%%%%%%%%%%%%%
\begin{corollary}\label{cor:finitetypeasmorhp}
%%%%%%%%%%%%%%%%%%%%%%%%%%%%%%%%%%%%%%%%%%%%%%%%%%%%%%%%%%%%%%%%%%%%%%%%%%
	Let $F$ be a rational invariant of homology spheres of finite type $k$. There is a well-defined map
	$$F:\mathcal{T}_{g,1}/\sqrt{\mathcal{T}_{g,1}(2k+1)}\rightarrow \mathbb{Q} $$  
	and the restriction of $F$ to $\sqrt{\mathcal{T}_{g,1}(m)}/\sqrt{\mathcal{T}_{g,1}(2k+1)}$
	is a homomorphism if $m>k$.
\end{corollary}
\begin{proof}
	Let $f \in \mathcal{T}_{g,1}$ and $\phi \in \sqrt{\mathcal{T}_{g,1}(2k+1)}$. Then by the above Proposition,
	\begin{align*}
	F(\mathbb{S}^3_{f\phi}) & = F(\mathbb{S}^3_{f}) + F(\mathbb{S}^3_{\phi}) \\
							& = F(\mathbb{S}^3_f) \textrm{ because } \phi \in   \sqrt{\mathcal{T}_{g,1}(2k+1)}.
	\end{align*}
The fact that the induced map on $\sqrt{\mathcal{T}_{g,1}(m)}/\sqrt{\mathcal{T}_{g,1}(2k+1)}$
is a homomorphism if $m>k$ is then trivial.
\end{proof}

Secondly, we have the Johnson filtration, defined as follows. Consider $N_k(\pi):=\pi/\Gamma_{k+1} \pi$, the $k$-th nilpotent quotient of $\pi$, the fundamental group of the complement of the interior of the marked disc in our surface $\Sigma_{g,1}$ pointed at the boundary.
We grade these quotients so that $N_1(\pi)= H_{1}(\Sigma_{g,1},\mathbb{Z}) =H$. The canonical map $\mathcal{M}_{g,1} \rightarrow \text{Aut}(\pi)$ induces group homomorphisms $\rho_k:\mathcal{M}_{g,1} \rightarrow \text{Aut}(N_k(\pi))$, and we set $\mathcal{M}_{g,1}(k) = \ker \rho_k$. This gives us a filtration of the mapping class group by normal subgroups, the Johnson filtration and by~\cite[Proposition 4.1]{mor-ab} this is a descending central series.

	In particular, for all $k \geq 1$,
it turns out that the restriction of $\rho_{k+1}$
 to $\mathcal{M}_{g,1}(k)$ can be viewed as a map $\tau_k: \mathcal{M}_{g,1}(k) \rightarrow H \otimes \mathcal{L}_{k+1}$, where $\mathcal{L}_{k+1}$ is the degree $k+1$ part of the graded free Lie algebra generated by $H$; the map $\tau_k$ is known as the ($k$-th) Johnson homomorphism. As the target group of $\tau_k$ is a torsion-free group, we have inclusions:
\[
\mathcal{T}_{g,1}(k) \subseteq \sqrt{\mathcal{T}_{g,1}(k)}  \subseteq \mathcal{M}_{g,1}(k).
\]

Since $\mathcal{M}_{g,1}(1) = \mathcal{T}_{g,1}$, the Johnson groups give a decreasing  filtration of the Torelli group, and we may restrict the Heegaard splitting map to each of the terms of the filtration to get  maps
\[
\lim_{g\to \infty}\mathcal{A}_{g,1}\backslash\mathcal{M}_{g,1}(k)/\mathcal{B}_{g,1}  \longrightarrow \mathcal{S},
\]
where two elements $\phi,\psi \in \mathcal{M}_{g,1}(k)$ are equivalent if and only if there exist maps $\xi_a \in \mathcal{A}_{g,1}$, $\xi_b \in \mathcal{B}_{g,1}$ such that
$\phi = \xi_a \psi \xi_b$, and the limit is taken over the stabilization maps induced on the double coset space.
By results of Morita \cite{mor} for $k=1,2$, Pitsch \cite{pitsch3} for $k=3$ and Faes \cite{faes2} for $k=4$, we have:

%%%%%%%%%%%%%%%%%%%%%%%%%%%%%%%%%%%%%%%%%%%%%%%%%%%%%%%%%%%%
\begin{proposition}
\label{prop:k=1234}
%%%%%%%%%%%%%%%%%%%%%%%%%%%%%%%%%%%%%%%%%%%%%%%%%%%%%%%%%%%%
For $k=1,2,3$ and $4$ the above map is a bijection.
\end{proposition}
%%%%%%%%%%%%%%%%%%%%%%%%%%%%%%%%%%%%%%%%%%%%%%%%%%%%%%%%%%%%

Moreover, setting $\mathcal{A}_{g,1}(k) = \mathcal{A}_{g,1} \cap \mathcal{M}_{g,1}(k)$ and $\mathcal{B}_{g,1}(k) = \mathcal{B}_{g,1} \cap \mathcal{M}_{g,1}(k)$, we have that, by \cite[Theorem 1]{pitsch} for $k=1$ and by \cite[Proposition 6.7]{faes1} for $k=2,3$, the equivalence relation on $\mathcal{M}_{g,1}(k)$ takes the form:

%%%%%%%%%%%%%%%%%%%%%%%%%%%%%%%%%%%%%%%%%%%%%%%%%%%%%%%%%%%%
\begin{proposition}\label{thm:Johnsonandhlgyspheres}
%%%%%%%%%%%%%%%%%%%%%%%%%%%%%%%%%%%%%%%%%%%%%%%%%%%%%%%%%%%%
For $k=1,2,3$, the Heegaard splitting map induces a bijection:
\[
\begin{array}{rcl}
\displaystyle{\lim_{g \to \infty}} (\mathcal{A}_{g,1}(k) \setminus \mathcal{M}_{g,1}(k)/\mathcal{B}_{g,1}(k))_{\mathcal{AB}_{g,1}} & \longrightarrow & \mathcal{S} \\
\phi & \longmapsto & \mathbb{S}^3_\phi.	
\end{array}
\]
More precisely, two maps $\phi,\psi \in \mathcal{M}_{g,1}(k)$ are equivalent if and only if there exist maps $\xi_a \in \mathcal{A}_{g,1}(k)$, $\xi_b \in \mathcal{B}_{g,1}(k)$ and $\mu \in \mathcal{AB}_{g,1}$ such that
\[
\phi = \mu \xi_a \psi \xi_b \mu^{-1}.
\]
\end{proposition}
%%%%%%%%%%%%%%%%%%%%%%%%%%%%%%%%%%%%%%%%%%%%%%%%%%%%%%%%%%%%

%%%%%%%%%%%%%%%%%%%%%%%%%%%%%%%%%%%%%%%%%%%%%%%%%%%%%%%%%%%%
\subsection{The rational Torelli and Johnson Lie algebras}\label{subsec:Johnsalgebra}
%%%%%%%%%%%%%%%%%%%%%%%%%%%%%%%%%%%%%%%%%%%%%%%%%%%%%%%%%%%%

The two filtrations we are considering on  the Torelli group $\mathcal{T}_{g,1}$, the Johnson filtration $\lbrace \mathcal{M}_{g,1}(k) \rbrace_k$ and the lower central series $\lbrace \mathcal{T}_{g,1}(k) \rbrace_k$, are graded. In~\cite[Proposition 4.1]{mor-ab} it was proved that for any $k,l \geq 1$,  $[\mathcal{M}_{g,1}(k),\mathcal{M}_{g,1}(l)] \subseteq \mathcal{M}_{g,1}(k+l)$, and basic commutator calculus, shows that the same is true  for the lower central series, i.e. for any $k,l \geq 1$, $[\mathcal{T}_{g,1}(k),\mathcal{T}_{g,1}(l)] \subseteq \mathcal{T}_{g,1}(k+l)$. We therefore have associated graded Lie algebras
\[
\text{Gr}^{\Gamma} \mathcal{T}_{g,1} = \bigoplus_{k=1}^\infty \mathcal{T}_{g,1}(k)/\mathcal{T}_{g,1}(k+1), \quad \text{Gr}^J \mathcal{T}_{g,1} = \bigoplus_{k=1}^\infty \mathcal{M}_{g,1}(k)/\mathcal{M}_{g,1}(k+1),
\]
and we denote by
\begin{equation*}
	\begin{aligned}
		\mathfrak{t}_{g,1}=\bigoplus_{k=1}^\infty\mathfrak{t}_{g,1}(k),\qquad & \mathfrak{t}_{g,1}(k)= (\mathcal{T}_{g,1}(k)/\mathcal{T}_{g,1}(k+1))\otimes \mathbb{Q},
	\end{aligned}
\end{equation*}
respectively
\begin{equation*}
	\begin{aligned}
		\mathfrak{m}_{g,1}=\bigoplus_{k=1}^\infty\mathfrak{m}_{g,1}(k),\qquad & \mathfrak{m}_{g,1}(k)= (\mathcal{M}_{g,1}(k)/\mathcal{M}_{g,1}(k+1))\otimes \mathbb{Q},
	\end{aligned}
\end{equation*}
their rationalization. The canonical inclusion $\mathcal{T}_{g,1}(k) \hookrightarrow \mathcal{M}_{g,1}(k)$ in each degree $k\geq 1$, gives us a canonical map of graded Lie algebras $\mathfrak{t}_{g,1} \rightarrow \mathfrak{m}_{g,1}$. It is a rather remarkable fact due to Hain~\cite{hain1}, that this map is surjective, but more importantly to us from \cite[Proposition 3.1 and Theorem 1.2]{MSS}, \cite[Theorem 1.2]{MSS} and \cite[Proposition 6.3]{mor_str} we learn that

%%%%%%%%%%%%%%%%%%%%%%%%%%%%%%%%%%%%%%%%%%%%%%%%%%%%%%%%
\begin{theorem}\label{teo:dif_tm}
%%%%%%%%%%%%%%%%%%%%%%%%%%%%%%%%%%%%%%%%%%%%%%%%%%%%%%%%
	For $k=3,4,5,6,$ the map $\mathfrak{t}_{g,1}(k) \rightarrow  \mathfrak{m}_{g,1}(k)$ is an isomorphism.
\end{theorem}
%%%%%%%%%%%%%%%%%%%%%%%%%%%%%%%%%%%%%%%%%%%%%%%%%%%%%%%%

This is highly non-trivial, and by the same arguments as in the proof of \cite[Corollary~1.3]{MSS} applied to the case of a surface with boundary we have 
%%%%%%%%%%%%%%%%%%%%%%%%%%%%%%%%%%%%%%%%%%%%%%%%%%%%%%%%
\begin{corollary}[\cite{MSS}]\label{cor:ker_d1}
%%%%%%%%%%%%%%%%%%%%%%%%%%%%%%%%%%%%%%%%%%%%%%%%%%%%%%%%
	For any $k=3,4,5,6,7$, the group $\mathcal{T}_{g,1}(k)$ is a finite index subgroup of the kernel of the non-trivial homomorphism
	\[
	\lambda:\mathcal{M}_{g,1}(k)\rightarrow \mathbb{Q}.
	\]
	induced by the Casson invariant.
\end{corollary}
%%%%%%%%%%%%%%%%%%%%%%%%%%%%%%%%%%%%%%%%%%%%%%%%%%%%%%%%

Since $\mathbb{Q}$ is torsion free,  equivalently this shows that for $k= 3,4,5,6$ and $7$
\[
\sqrt{\mathcal{T}_{g,1}(k)} = \ker \lambda\vert_{\mathcal{M}_{g,1}(k)}.
\]

%%%%%%%%%%%%%%%%%%%%%%%%%%%%%%%%%%%%%%%%%%%%%%%%%%%%%%%%%%%%
%%%%%%%%%%%%%%%%%%%%%%%%%%%%%%%%%%%%%%%%%%%%%%%%%%%%%%%%%%%%
\subsection{Representation theory for $GL_g(\mathbb{Z})$}\label{sec:TensorandTrees}
%%%%%%%%%%%%%%%%%%%%%%%%%%%%%%%%%%%%%%%%%%%%%%%%%%%%%%%%%%%%
%%%%%%%%%%%%%%%%%%%%%%%%%%%%%%%%%%%%%%%%%%%%%%%%%%%%%%%%%%%%
In~\cite{levine1} Garoufalidis and Levine lay out a graph interpretation of the representation theory for $GL_g(\mathbb{Q})$ and $\Sp_{2g}(\mathbb{Q})$, which is very convenient for computations. We will adapt here some of their results to $GL_g(\mathbb{Z}) \subseteq GL_g(\mathbb{Q})$, as this is the group that encodes the homological action of $\mathcal{AB}_{g,1}$. The result we are interested in is a variant of~\cite[Proposition 2.18]{levine1}. Recall that we have fixed a symplectic basis $\{a_i,b_i\}_{1 \leq i \leq g}$ of $H$, and hence of $H_\mathbb{Q}=H \otimes \mathbb{Q}$. We will call an elementary tensor $u_1\otimes u_2 \dots \otimes u_p \in  \otimes^p H_\mathbb{Q}$ a \emph{basic tensor} if all its entries $u_i$ belong to our preferred symplectic basis.

%%%%%%%%%%%%%%%%%%%%%%%%%%%%%%%%%%%%%%%%%%%%%%%%%%%%%%%%%%%%
\begin{proposition}\label{prop:chords}
%%%%%%%%%%%%%%%%%%%%%%%%%%%%%%%%%%%%%%%%%%%%%%%%%%%%%%%%%%%%
For any  integer $1 \leq n < g$, the coinvariant quotient $(\otimes^{2n} H_\mathbb{Q})_{GL_g(\mathbb{Z})}$ is generated by the images of the basic tensors  for which  for each $1 \leq  i \leq g$, either the pair of elements $\{a_i,b_i\}$ appears exactly once or does not appear at all.  Moreover, all  basic tensors for which the number of appearances of $a_i$ is different to that of $b_i$ for some $i$ are $0$ in the coinvariant quotient.
\end{proposition}
%%%%%%%%%%%%%%%%%%%%%%%%%%%%%%%%%%%%%%%%%%%%%%%%%%%%%%%%%%%%

\begin{proof}

The basic tensors form a basis  for  $(\otimes^{2n} H_\mathbb{Q})$ and hence their images certainly generate $(\otimes^{2n} H_\mathbb{Q})_{GL_g(\mathbb{Z})}$. The canonical action of the symmetric group $\mathfrak{S}_{k}$ by permuting the factors of $\otimes^kH_\mathbb{Q}$ commutes with the $GL_g(\mathbb{Z})$ action and hence induces an action on the coinvariant quotient $(\otimes^{2n} H_\mathbb{Q})_{GL_g(\mathbb{Z})}$. We will call any two elements that are related by such a permutation \emph{permutation equivalent}. In particular, any basic tensor is permutation equivalent to an ordered basic tensor, of the form
\[
(\otimes^{p_1}a_1)\otimes (\otimes^{q_1}b_1)\otimes \cdots \otimes (\otimes^{p_g}a_g)\otimes (\otimes^{q_g}b_g),
\]
where $p_i,q_i$ are non-negative integers with $\sum_{i=1}^{g}(p_i+q_i)=2n$.

\begin{itemize}	
\item[\textbf{Step 1.}] We claim that any basic tensor that is permutation equivalent to a basic ordered tensor for which for some $1 \leq j \leq g$, $p_j+q_j$ is odd has a trivial image in $(\otimes^{2n} H_\mathbb{Q})_{GL_g(\mathbb{Z})}$. Indeed, the elementary matrix $D_j(-1) \in GL_g(\mathbb{Z})$, which is the diagonal matrix with all entries equal to $1$ but the $j^{th}$ which is equal to $-1$, acts on the basic tensor by multiplication by $(-1)^{p_j+q_j}=-1$, hence in the coinvariant quotient this element has to be $0$.
\end{itemize}

Because $n<g$ and $\sum_{i=1}^{g}(p_i+q_i)=2n$, the fact that $p_i + q_i$ is even imposes that $p_i+q_i =0$ for some $1 \leq i \leq g$. Up to a permutation of the indices, we may assume that $i=g$. To make the notation lighter in all what follows we set $R :=\bigotimes_{i=2}^{g-1}((\otimes^{p_i}a_i)\otimes (\otimes^{q_i}b_i))$, so any basic tensor is zero in the coinvariant quotient or permutation equivalent to
\[
(\otimes^{p_1}a_1)\otimes (\otimes^{q_1}b_1)\otimes R,
\qquad \text{with } p_i + q_i \text{ even for all } i. 
\]

\begin{itemize}
\item[\textbf{Step 2.}] We claim that the basic tensors permutation equivalent to one for which   $\lbrace p_i=0,\;q_i> 0\rbrace$ or $\lbrace q_i=0,\; p_i> 0\rbrace$ for some $i$ are sent to $0$ in  $(\otimes^{2n} H_\mathbb{Q})_{GL_g(\mathbb{Z})}$. We prove this for $q_1=0,\; p_1>0$, i.e. for a basic tensor permutation equivalent to $(\otimes^{p_1}a_1)\otimes R$, the other cases are analogous.

Consider the basic tensor
\[
a_{g}\otimes(\otimes^{p_1-1}a_1)\otimes R,
\]
and act on it by the elementary matrix $Id + E_{1,g} = L_{1,g}(1)\in GL_g(\mathbb{Z})$. By Step~1, the element is trivial in  $(\otimes^{2n} H_\mathbb{Q})_{GL_g(\mathbb{Z})}$, so in this quotient we have
 \begin{align*}
 	0 &= a_{g}\otimes(\otimes^{p_1-1}a_1)\otimes R \\
 	& = L_{1,g}(1) (a_{g}\otimes(\otimes^{p_1-1}a_1)\otimes R) \\
 	& = (a_g + a_1)\otimes(\otimes^{p_1-1}a_1)\otimes R \\
 	& = a_g\otimes(\otimes^{p_1-1}a_1)\otimes R + a_1^{p_1} \otimes R \\
 	& = a_1^{p_1} \otimes R.
 \end{align*}

\item[\textbf{Step 3.}] We claim that the basic tensors that are permutation equivalent to one for which $\lbrace p_i=1,\;q_i> 1\rbrace$ or $\lbrace q_i=1,\; p_i>1\rbrace$ for some index $i$, are sent to $0$ in  $(\otimes^{2n} H_\mathbb{Q})_{GL_g(\mathbb{Z})}$. We prove this for $p_1=1<q_1$, so for basic tensors permutation equivalent to $a_1\otimes (\otimes^{q_1}b_1)\otimes R$, the other cases are analogous. 

We act on the basic tensor
\[
a_{g} \otimes(\otimes^{q_1-1}b_1)\otimes b_{g}\otimes R
\] 
by the elementary matrix $L_{1,g}(-1) = Id -E_{1,g}$. By Step 2, in  $(\otimes^{2n} H_\mathbb{Q})_{GL_g(\mathbb{Z})}$ this element is $0$, hence in the coinvariant quotient:
\[
\begin{aligned}
	0 & =a_{g} \otimes(\otimes^{q_1-1}b_1)\otimes b_{g}\otimes R  \\
	& = L_{1,g}(-1)(a_{g} \otimes(\otimes^{q_1-1}b_1)\otimes b_{g}\otimes R) \\
	& = (a_{g}-a_1) \otimes(\otimes^{q_1-1}(b_1+b_g)\otimes b_{g}) \otimes R \\
	& = a_g \otimes (\otimes^{q_1-1} (b_1 +b_g)\otimes  b_g)\otimes R -a_1  \otimes (\otimes^{q_1-1} (b_1+b_g)\otimes  b_g) \otimes R.
\end{aligned}
\]
As we expand the second summand by linearity,
we will get basic tensors with at least one $b_g$ but no $a_g$,
hence by Step 2, all these elements are $0$ in the coinvariant quotient.
We now expand the first summand by linearity; we  get basic tensors with $(q_1-1-k)$ basis elements with $b_1$ but no $a_1$ and $k+1$ basis elements $b_g$ for $0 \leq k \leq q_1-1$. By Step 2, unless $q_1-1=k$, all this basic tensors are $0$ in the coinvariant quotient. Hence all that remains in the last line is $a_{1} \otimes(\otimes^{q_1}b_{1})\otimes R$, as we wanted.\\

\item[\textbf{Step 4.}] We claim that the basic tensors that are permutation equivalent to one for which  $1<p_i,q_i$ are  equivalent to a sum of elements each one permutation equivalent to
\[
(\otimes^{p_1-1}a_1)\otimes (\otimes^{q_1-1}b_1)\otimes R \otimes a_{g}\otimes b_{g}.
\]

Consider the element
\[
a_{g}\otimes(\otimes^{p_1-1}a_1)\otimes (\otimes^{q_1}b_1)\otimes R,
\]
which is $0$ in $(\otimes^{2n} H_\mathbb{Q})_{GL_g(\mathbb{Z})}$ by Step $1$. We act on it by $L_{g,1}(-1)$, to get
\begin{align*}
	0  = &a_{g}\otimes(\otimes^{p_1-1}a_1)\otimes (\otimes^{q_1}b_1)\otimes R \\
	 = &L_{g,1}(-1)(a_{g}\otimes(\otimes^{p_1-1}a_1)\otimes (\otimes^{q_1}b_1)\otimes R) \\
	 = &(a_g-a_1)\otimes(\otimes^{p_1-1}a_1)\otimes (\otimes^{q_1}(b_1+b_g))\otimes R \\
	 =& a_g\otimes(\otimes^{p_1-1}a_1)\otimes (\otimes^{q_1}(b_1+b_g))\otimes R \\ & -(\otimes^{p_1}a_1)\otimes (\otimes^{q_1}(b_1+b_g))\otimes R.
\end{align*}
As we expand the first summand from the last line we get elements that are permutation equivalent to  $a_g \otimes(\otimes^k b_g) \otimes (\otimes^{p_1-1}a_1)\otimes (\otimes^{q_1-k}b_1) \otimes R$, for $0 \leq k \leq q_1$.  By Step 3, all these elements are $0$ but  for $k=1$, and there are $q_1$ elements of this form.

In the expansion of the second summand we get basic tensors permutation equivalent to $(\otimes^{p_1}a_1)\otimes (\otimes^{q_1-k}b_1)\otimes(\otimes^k b_g))\otimes R$ for $0 \leq k \leq q_1$. By Step 2 all these elements are $0$ but for  $k=0$.

In the end any basic tensor that is permutation equivalent to $(\otimes^{p_1}a_1)\otimes (\otimes^{q_1}b_1)\otimes R$ is  equal in $(\otimes^{2n} H_\mathbb{Q})_{GL_g(\mathbb{Z})}$ to the sum $q_1$ basic tensors each one permutation equivalent to  $(\otimes^{p_1-1}a_1)\otimes (\otimes^{q_1-1}b_1)\otimes R \otimes a_{g}\otimes b_{g}$.

\end{itemize}

Applying recursively Step 4, we conclude firstly, that the basic tensors that are permutation equivalent to  $\bigotimes_{i=1}^s((\otimes^{p_i}a_i)\otimes (\otimes^{q_i}b_i))$ with $1<p_i<q_i$ or $1<q_i<p_i$ for some index $i$ are also zero in the coinvariants module $(\otimes^{2n} H_\mathbb{Q})_{GL_g(\mathbb{Z})}$  and secondly that the family of basic tensors that are permutation equivalent to  
\[
a_1\otimes b_1 \otimes a_2 \otimes b_2\otimes \cdots \otimes a_n\otimes b_n.
\]
form a generating set of $(\otimes^{2n} H_\mathbb{Q})_{GL_g(\mathbb{Z})}$.

\end{proof}

Let $\mathcal{A}_\ast(H_\mathbb{Q})$ be the Lie algebra of uni-trivalent trees with cyclically oriented trivalent vertices that are labelled by elements of $H_\mathbb{Q}$ modulo linearity in the labels and  the so-called IHX and  AS relations introduced by Garoufalidis and Ohtsuki~\cite{MR1489202}. This algebra gives a convenient description of the target of the Johnson homomorphisms. Moreover, by Sakasai~\cite{saka}, the Johnson homomorphisms induce isomorphisms in low degrees,   $\mathfrak{m}_{g,1}(1)\simeq \mathcal{A}_1(H_\mathbb{Q})$ and $\mathfrak{m}_{g,1}(2)\simeq \mathcal{A}_2(H_\mathbb{Q})$. In more concrete terms, an element  $ a \wedge b \wedge c \in \mathfrak{m}_{g,1}(1) \simeq \Lambda^3H_\mathbb{Q}$ corresponds to the tree:
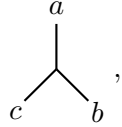
\[
\Ytree{a}{b}{c},
\]
and a typical element in $\mathcal{M}_{g,1}(2)$ is a sum of $H$-shaped trees:
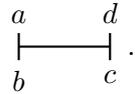
\[
\tree{a}{b}{c}{d}.
\] 

Let $S^2(\Lambda^2 H_\mathbb{Q})$ denote the second symmetric power of $\Lambda^2 H_\mathbb{Q}$; notice that the vector space $\Lambda^4 H_\mathbb{Q}$ can be embedded in $S^2(\Lambda^2 H_\mathbb{Q})$ by sending $a\wedge b\wedge c \wedge d$ to
$(a\wedge b)(c \wedge d)-(a\wedge c)(b \wedge d)+(a\wedge d)(b \wedge c)$.  The final link that gives the appropriate context to  apply the representation theory we described is provided by:

%%%%%%%%%%%%%%%%%%%%%%%%%%%%%%%%%%%%%%%%%%%%%%%%%%%%%%%%%%%%%%%%%%%
\begin{proposition}[Proposition 3.1 in \cite{MB}]
\label{prop:tree-deriv1}
%%%%%%%%%%%%%%%%%%%%%%%%%%%%%%%%%%%%%%%%%%%%%%%%%%%%%%%%%%%%%%%%%%%
There is an isomorphism of $GL_g(\mathbb{Z})$-modules:
\[
\xymatrix@C=10mm@R=10mm{\mathcal{A}_2(H_\mathbb{Q})\ar@{->}[r]^-{\sim} & \dfrac{S^2(\Lambda^2 H_\mathbb{Q})}{\Lambda^4 H_\mathbb{Q}}},
\]
which sends the tree $\tree{a}{b}{c}{d}$ to the element $(a\wedge b)(c \wedge d)$.
In particular for $g \geq 3$,  $\mathfrak{m}_{g,1}(2)_{GL_g(\mathbb{Z})} \simeq \mathcal{A}_2(H_\mathbb{Q})_{GL_g(\mathbb{Z})}$ is generated by any two of the trees:
\[
\tree{a_1}{b_1}{b_2}{a_2}, \quad \tree{a_1}{a_2}{b_2}{b_1} \quad \text{and}\quad \tree{a_1}{b_2}{a_2}{b_1}.
\]
\end{proposition}
%%%%%%%%%%%%%%%%%%%%%%%%%%%%%%%%%%%%%%%%%%%%%%%%%%%%%%%%%%%%%%%%%%%
\begin{proof}
	The only part that is not covered by \cite[Proposition 3.1]{MB} is the description of the generators. The first part shows that $\mathcal{A}_2(H_\mathbb{Q})$ is a quotient of $\otimes^4H_\mathbb{Q}$. Then, taking $GL_g(\mathbb{Z})$-coinvariants we get that $\mathcal{A}_2(H_\mathbb{Q})_{GL_g(\mathbb{Z})}$ is a quotient of $(\otimes^4H_\mathbb{Q})_{GL_g(\mathbb{Z})}$ and hence, by Proposition \ref{prop:chords}, $\mathcal{A}_2(H_\mathbb{Q})_{GL_g(\mathbb{Z})}$ is generated by the trees that are images of basic elements permutation equivalent to  $a_1 \otimes b_1 \otimes a_2 \otimes b_2$. The symmetries of the trees imply that we can always assume that the top left corner is in fact $a_1$, and then the remaining $6$ trees are pairwise identified up to a non trivial scalar:
\[
	\begin{array}{c}
		\tree{a_1}{b_1}{b_2}{a_2}  = - \tree{a_1}{b_1}{a_2}{b_2}, \qquad
	 	\tree{a_1}{a_2}{b_2}{b_1} = - \tree{a_1}{a_2}{b_1}{b_2}, \\
	 	 \tree{a_1}{b_2}{a_2}{b_1} = -\tree{a_1}{b_2}{b_1}{a_2}.
	\end{array}
\]
Finally, by the IHX relation, 
\[
\tree{a_1}{b_1}{b_2}{a_2} = \tree{a_1}{a_2}{b_1}{b_2} - \tree{a_1}{b_2}{b_1}{a_2}.
\]
\end{proof}

%%%%%%%%%%%%%%%%%%%%%%%%%%%%%%%%%%%%%%%%%%%%%%%%%%%%%%%%%%%%
%%%%%%%%%%%%%%%%%%%%%%%%%%%%%%%%%%%%%%%%%%%%%%%%%%%%%%%%%%%%
\subsection{The handlebody groups in $\mathfrak{m}_{g,1}(2)$}\label{sec:Lagtraces}
%%%%%%%%%%%%%%%%%%%%%%%%%%%%%%%%%%%%%%%%%%%%%%%%%%%%%%%%%%%%
%%%%%%%%%%%%%%%%%%%%%%%%%%%%%%%%%%%%%%%%%%%%%%%%%%%%%%%%%%%%

To control the behaviour of an invariant when restricted to subgroups in the Johnson filtration it is essential to have a good understanding of the images of the groups $\mathcal{A}_{g,1}(k)$, $\mathcal{B}_{g,1}(k)$ and $\mathcal{AB}_{g,1}(k)$. In our case, for $k=2$, this is provided by the Lagrangian traces introduced by Faes~\cite{faes1}, which we now  briefly explain.
By expanding any vector with respect to our preferred  basis for $H_\mathbb{Q}$, we find that the elements in $\mathcal{A}_2(H_\mathbb{Q})$ are sums of elements whose $4$ labels are elements in the basis $a_i,b_i$.
Given $k,l$ non-negative integers with $k+l=4,$ we denote by $W(a^kb^l)$ the vector subspace of $\mathfrak{m}_{g,1}(2)$ generated by those trees with $a$'s in $k$ labels and $b$'s in $l$ labels. In~\cite[Lemma 4.1 and Proposition 3.8]{faes1}, Q. Faes showed that  $\tau_2(\mathcal{A}_{g,1}(2)) \subseteq W(a^4)\oplus W(a^3b)\oplus W(a^2b^2)\oplus W(ab^3)$. Moreover, on this direct sum $W(a^{\geq 1}b)$, he defined 
a trace-like operator called $\Tr^A: W(a^{\geq 1}b) \rightarrow S^2(B)$, 
and  proved in  \cite[Theorem 5.1]{faes1} that
\[
\tau_2(\mathcal{A}_{g,1}(2))=\ker(\Tr^A)\cap \im(\tau_2).
\]
There is an analogous trace-like $\Tr^B: W(ab^{\geq 1}) \rightarrow S^2(A)$ for which the same proof shows that
\[
\tau_2(\mathcal{B}_{g,1}(2))=\ker(\Tr^B)\cap \im(\tau_2).
\]
The trace $\Tr^A$ (and analogously $\Tr^B$) can be explicitly computed on trees (\cite[Example 4.5]{faes1})  as follows. Let $a \in A$, $c,d,e \in H$, and denote by $p_B$ the projection on $B$ parallel to $A$, then
\[
\Tr^A\Big(\tree{a}{c}{d}{e}\Big) = \omega(a,e)p_B(d)p_B(c) - \omega(a,d)p_B(e)p_B(c).
\]
And exchanging the roles of $A$ and $B$, for $b \in B$, $c,d,e \in H$
\[
\Tr^B\Big(\tree{b}{c}{d}{e}\Big) = \omega(b,e)p_A(d)p_A(c) - \omega(b,d)p_A(e)p_A(c).
\]
Moreover by \cite[Corollary A.4]{faes1} we have that
\[
\tau_2(\mathcal{AB}_{g,1}(2))\otimes \mathbb{Q}=(\tau_2(\mathcal{A}_{g,1}(2))\otimes \mathbb{Q}) \cap (\tau_2(\mathcal{B}_{g,1}(2))\otimes \mathbb{Q}).
\]
From these facts, an easy computation shows, setting  $W_0(ab^3)=Ker(\Tr^A)\cap W(ab^3)$ and $W_0(a^3b)=Ker(\Tr^B)\cap W(a^3b)$, that
\begin{align*}
	\tau_2(\mathcal{A}_{g,1}(2)) \otimes \mathbb{Q} & =W(a^4) \oplus  W(a^3b) \oplus W(a^2b^2)\oplus W_0(ab^3), \\
		\tau_2(\mathcal{B}_{g,1}(2)) \otimes \mathbb{Q} & =  W_0(a^3b) \oplus W(a^2b^2)\oplus W(ab^3) \oplus W(b^4),\\
\tau_2(\mathcal{AB}_{g,1}(2)) \otimes  \mathbb{Q} & =  W_0(a^3b) \oplus W(a^2b^2)\oplus W_0(ab^3).
\end{align*}

%%%%%%%%%%%%%%%%%%%%%%%%%%%%%%%%%%%%%%%%%%%%%%%%%%%%%%%%%%%%
%%%%%%%%%%%%%%%%%%%%%%%%%%%%%%%%%%%%%%%%%%%%%%%%%%%%%%%%%%%%
\section{Invariants and their trivialized 2-cocycles}\label{sec:Invandcocycles}
\label{sec:InvTrivcoc}
%%%%%%%%%%%%%%%%%%%%%%%%%%%%%%%%%%%%%%%%%%%%%%%%%%%%%%%%%%%%
%%%%%%%%%%%%%%%%%%%%%%%%%%%%%%%%%%%%%%%%%%%%%%%%%%%%%%%%%%%%

Let us go back to  the  bijection 
\begin{equation}
	\label{bij_S3_2}
	\begin{tikzcd}
		\displaystyle\lim_{g\to \infty}(\mathcal{A}_{g,1}(2)\backslash\mathcal{M}_{g,1}(2)/\mathcal{B}_{g,1}(2))_{\mathcal{AB}_{g,1}} \ar[r,"\sim"] & \mathcal{S}.
	\end{tikzcd}
\end{equation}
 Let $F: \mathcal{S} \longrightarrow \mathbb{Q}$ be a rational valued invariant of integral homology spheres, and assume without loss of generality that $F$ is normalized, i.e., $F(\mathbb{S}^3) =0$. The invariant $F$ determines and is determined by the sequence of maps for $g \geq 1$, 
 \[
 	\begin{tikzcd}
 F_g: \mathcal{M}_{g,1}(2) \ar[r] & 	\displaystyle\lim_{g\to \infty}(\mathcal{A}_{g,1}(2)\backslash\mathcal{M}_{g,1}(2)/\mathcal{B}_{g,1}(2))_{\mathcal{AB}_{g,1}} \ar[r,"\sim"] & \mathcal{S} \ar[r, "F"] & \mathbb{Q}.
 \end{tikzcd}
 \]
 
 Because $F$ is an invariant,  the maps $F_g$ satisfy the following conditions:

 \begin{enumerate}[$i)$]
 	\item Stability: $F_{g+1}(\phi)=F_g(\phi) \quad \text{for every } \phi\in \mathcal{M}_{g,1}(2)$,
 	\item Double class condition: $F_g(\xi_a \phi \xi_b)=F_g(\phi) \quad \text{for every } \phi\in \mathcal{M}_{g,1}(2),\;\xi_a\in \mathcal{A}_{g,1}(2),\;\xi_b\in \mathcal{B}_{g,1}(2)$,
 	\item Conjugacy Invariance: $F_g(\mu \phi \mu^{-1})=F_g(\phi)  \quad \text{for every }  \phi \in \mathcal{M}_{g,1}(2), \; \mu\in \mathcal{AB}_{g,1}$.
 \end{enumerate}

Because the stabilization map $\mathcal{M}_{g,1} \rightarrow \mathcal{M}_{g+1,1}$ is injective, property $i)$ shows that to control $F$ it is enough to control $F_g$ for $g$ large enough.

In order to actually compute an invariant, a very useful tool is a surgery formula, the prototypical one being that of the Casson invariant. Surgery formulas have an algebraic counter part, the trivialized $2$-cocycles associated to an invariant. To our sequence of maps $(F_g)_{g \geq 1}$, with domain $\mathcal{M}_{g,1}(2)$,  we associate the sequence of maps:
 \begin{align*}
 	C_g: \mathcal{M}_{g,1}(2)\times \mathcal{M}_{g,1}(2) & \longrightarrow \mathbb{Q}, \\
 	(\phi,\psi) & \longmapsto F_g(\phi)+F_g(\psi)-F_g(\phi\psi).
 \end{align*}
 This is a  family of trivialized $2$-cocycles and from the invariant $F$ they inherit the following properties:
 \begin{enumerate}[(1)]
 	\item Stability:  $C_{g+1}|_{\mathcal{M}_{g,1}(2) \times \mathcal{M}_{g,1}(2)} = C_g$, 
 	\item Double class condition: If $\phi\in \mathcal{A}_{g,1}(2)$ or $\psi \in \mathcal{B}_{g,1}(2)$ then $C_g(\phi, \psi)=0$, 
 	\item Conjugacy Invariance: For all $\mu \in \mathcal{AB}_{g,1}$ and for all $\phi,\psi \in \mathcal{M}_{g,1}(2)$,
 	\[
 	C_g(\mu\phi\mu^{-1},\mu\psi\mu^{-1}) = C_g(\phi,\psi).
 	\]
 \end{enumerate}
  We  want to introduce into this picture the finite type condition. Unfortunately we do not know how to characterize this condition purely in group theoretical terms. We will hence explore the consequences of   conditions stated in Proposition~\ref{prop:coc-finite-type}. Fix an integer $k$, the degree of the invariant you are interested into, then the 2-cocycle $C_g$ satisfies:
 
 \begin{enumerate}
 	\item[(4)] Weak finite type $k$. The $2$-cocycles $\{C_g\}_g$ are zero on $\sqrt{\mathcal{T}_{g,1}(k+1)}\times \sqrt{\mathcal{T}_{g,1}(k)}$ and $\sqrt{\mathcal{T}_{g,1}(k)}\times \sqrt{\mathcal{T}_{g,1}(k+1)}$.
 \end{enumerate}

The maps
\[\begin{aligned}
& \begin{array}{rcl}
\left\{ \mathbb{Q}-\text{valued degree } k \text{ invariants} \right\} & \longrightarrow& \left\{ \text{series of maps } (F_g)_{g \geq 1} \text{ satisfying } i)\text{ - }iii) \right\}, 
\end{array}
\\
& \begin{array}{rcl}
\left\{ \mathbb{Q}-\text{valued degree } k \text{ invariants} \right\} & \longrightarrow& \left\{
\begin{array}{c}
\text{series of $2$-cocycles } (C_g)_{g \geq 1}  \\
\text{ satisfying } (1)\text{-}(4) \end{array}
\right\}
\end{array}
\end{aligned}
\]
are clearly linear. The kernel of the second map are those invariants whose associated functions on $\mathcal{M}_{g,1}(2)$  are homomorphisms. By a result of Morita~\cite{mor1} the Casson invariant is such an invariant, and it turns to be essentially the only one.

%%%%%%%%%%%%%%%%%%%%%%%%%%%%%%%%%%%%%%%%%%%%%%%%%%%%%%%%%%%%
\begin{theorem}\label{thm:CassonHomomorphInv}
%%%%%%%%%%%%%%%%%%%%%%%%%%%%%%%%%%%%%%%%%%%%%%%%%%%%%%%%%%%%
Up to a multiplicative constant, the Casson invariant is the unique $\mathbb{Q}$-valued invariant whose associated functions on $\mathcal{M}_{g,1}(2)$ are homomorphisms.
\end{theorem}
%%%%%%%%%%%%%%%%%%%%%%%%%%%%%%%%%%%%%%%%%%%%%%%%%%%%%%%%%%%%
 
 The key to prove this statement is the computation of
 \[
 \Hom(\mathcal{M}_{g,1}(2);\mathbb{Q})^{\mathcal{AB}_{g,1}}  \simeq \Hom(H_1\left(\mathcal{M}_{g,1}(2);\mathbb{Q}\right)_{\mathcal{AB}_{g,1}};\mathbb{Q}),
 \]
and it starts from the following result of Massuyeau-Faes~\cite{MF}. As we will need a slightly different description as the one stated in their work for our computations we will sketch the proof of the result.

%%%%%%%%%%%%%%%%%%%%%%%%%%%%%%%%%%%%%%%%%%%%%%%%%%%%%%%%%%%%%%%%%%%%%%%%% 
 \begin{proposition}\cite[Theorem 3.2]{MF}\label{prop:ab_M(2)/T(5)}
 %%%%%%%%%%%%%%%%%%%%%%%%%%%%%%%%%%%%%%%%%%%%%%%%%%%%%%%%%%%%%%%%%%%%%%%%% 
 	For $g \geq 6$, there is an $\mathcal{AB}_{g,1}$-equivariant isomorphism of $\mathbb{Q}$-vector spaces:
 	\[
 	H_1(\mathcal{M}_{g,1}(2);\mathbb{Q})  \simeq \mathbb{Q}\oplus \mathfrak{m}_{g,1}(2)\oplus \mathfrak{m}_{g,1}(3),
 	\]
 	where only the subspace $\mathbb{Q}$ and the quotient $\mathfrak{m}_{g,1}(2)$ are canonical.
 \end{proposition}
%%%%%%%%%%%%%%%%%%%%%%%%%%%%%%%%%%%%%%%%%%%%%%%%%%%%%%%%%%%%

 \begin{proof}
 The proof consists in analysing the rationalisation of the following  filtration of the abelianization of $\mathcal{M}_{g,1}(2)$:
 	\[
  \dfrac{\mathcal{M}_{g,1}(4)}{[\mathcal{M}_{g,1}(2),\mathcal{M}_{g,1}(2)]} \subseteq \dfrac{\mathcal{M}_{g,1}(3)}{[\mathcal{M}_{g,1}(2),\mathcal{M}_{g,1}(2)]} \subseteq \dfrac{\mathcal{M}_{g,1}(2)}{[\mathcal{M}_{g,1}(2),\mathcal{M}_{g,1}(2)]}.
 	\]
 	Notice that  $g \geq 3$, the Johnson group $\mathcal{M}_{g,1}(2)$ is finitely generated by~\cite{MR4416592}, hence this is a filtration of finitely generated abelian groups.

 The key point is to show that the group  $[\mathcal{M}_{g,1}(2),\mathcal{M}_{g,1}(2)]$ is a finite index subgroup of the kernel on the map $\lambda: \mathcal{M}_{g,1}(4) \rightarrow \mathbb{Q}$ induced by the Casson invariant. 
 In the boundary-less case this is a consequence of~\cite[Corollary 1.4]{MSS}, but the arguments from this work cannot be simply transposed to the case with a boundary component. By~\cite[Theorem 3.2]{MF}, the Casson invariant together with an extended version of the Johnson homomorphisms due to Heap~\cite{Heap}, $(\lambda,\widetilde{\tau}_2): \mathcal{M}_{g,1}(2) \rightarrow \mathbb{Z} \oplus \mathfrak{m}_{g,1}(2) \oplus \mathfrak{m}_{g,1}(3)$ induce an embedding of the torsion-free part of the group $H_1(\mathcal{M}_{g,1}(2),\mathbb{Z})$. 
 As a consequence $[\mathcal{M}_{g,1}(2),\mathcal{M}_{g,1}(2)]$ is a finite index subgroup of $\ker (\lambda,\tilde{\tau}_2)$. Again by~\cite{Heap}, we have $\ker \tilde{\tau}_2= \mathcal{M}_{g,1}(4)$. Finally, since the Casson invariant is a group homomorphism on $\mathcal{M}_{g,1}(2)$, we end up with a commutative diagram of rational vector spaces
\begin{equation}
\label{eq:diag-ab}
\begin{tikzcd}
	& & 0 \ar[d] & &  \\
	0 \ar[r] &   \mathbb{Q} \ar[r]& \dfrac{\mathcal{M}_{g,1}(3)}{[\mathcal{M}_{g,1}(2),\mathcal{M}_{g,1}(2)]} \otimes \mathbb{Q} \ar[r,"\tau_3"] \ar[d]  & \mathfrak{m}_{g,1}(3) \ar[r] & 0 \\
	& &  H_1(\mathcal{M}_{g,1}(2);\mathbb{Q}) \ar[ul, "\lambda"]  \ar[d, "\tau_2"]&  &   \\
	& & \mathfrak{m}_{g,1}(2) \ar[d]  & & \\
&	& 0 & & 
\end{tikzcd}
\end{equation}
By construction all the labelled arrows in this  diagram are $\mathcal{AB}_{g,1}$-equivariant, so this is a diagram of $\mathcal{AB}_{g,1}$-modules. In particular the horizontal extension is a split extension of $\mathcal{AB}_{g,1}$-modules.
\end{proof}

%%%%%%%%%%%%%%%%%%%%%%%%%%%%%%%%%%%%%%%%%%%%%%
 \begin{proposition}\label{prop_ab_K/T}
%%%%%%%%%%%%%%%%%%%%%%%%%%%%%%%%%%%%%%%%%%%%%%
 	The Casson invariant together with the second Johnson homomorphism $\tau_2$ induce an isomorphism for $g\geq 6.$
 	$$H_1(\mathcal{M}_{g,1}(2);\mathbb{Q})_{\mathcal{AB}_{g,1}}\simeq \mathbb{Q}\oplus (\mathfrak{m}_{g,1}(2))_{GL_g(\mathbb{Z})}.$$
 \end{proposition}
 %%%%%%%%%%%%%%%%%%%%%%%%%%%%%%%%%%%%%%%%%%%%%% 
 \begin{proof}
 	We first analyse the $\mathcal{AB}_{g,1}$-coinvariants in the horizontal extension in diagram \eqref{eq:diag-ab}. Since this extension is indeed a split extension of $\mathcal{AB}_{g,1}$-modules, taking the coinvariants preserve left exactness and hence we get another extension:
 	\[
 	\begin{tikzcd}
 		0 \ar[r] & \mathbb{Q} \ar[r] & \left(\dfrac{\mathcal{M}_{g,1}(3)}{[\mathcal{M}_{g,1}(2),\mathcal{M}_{g,1}(2)]} \otimes \mathbb{Q}\right)_{\mathcal{AB}_{g,1}} \ar[r, "\tau_2"] & \mathfrak{m}_{g,1}(3)_{\mathcal{AB}_{g,1}} \ar[r] & 0.
 	\end{tikzcd}
 	\]
 	The action of $\mathcal{AB}_{g,1}$ on $\mathfrak{m}_{g,1}(3)$ factors through $GL_g(\mathbb{Z}) \subseteq \Sp_{2g}(\mathbb{Z})$, and the matrix $-Id$ acts as $-Id$ on this $\mathbb{Q}$-vector space, hence the coinvariants have to be trivial and the exact sequence boils down to an isomorphism:
 	\[
 	\begin{tikzcd}
 		\mathbb{Q}  & \ar[l, "\lambda", "\sim"'] \left(\dfrac{\mathcal{M}_{g,1}(3)}{[\mathcal{M}_{g,1}(2),\mathcal{M}_{g,1}(2)]} \otimes \mathbb{Q}\right)_{\mathcal{AB}_{g,1}}.
 	\end{tikzcd}	
 	\]
 	
 	When computing the $\mathcal{AB}_{g,1}$-coinvariants of the vertical extension, a priori we only have an exact sequence:
 	\[
 	\begin{tikzcd}
 		  \left(\dfrac{\mathcal{M}_{g,1}(3)}{[\mathcal{M}_{g,1}(2),\mathcal{M}_{g,1}(2)]} \otimes \mathbb{Q}\right)_{\mathcal{AB}_{g,1}} \ar[r]  & \left(\dfrac{\mathcal{M}_{g,1}(2)}{[\mathcal{M}_{g,1}(2),\mathcal{M}_{g,1}(2)]}\otimes \mathbb{Q}\right)_{\mathcal{AB}_{g,1}} \arrow[dl,out=350,in=170,overlay, "\tau_2"'] \\
 \mathfrak{m}_{g,1}(2)_{\mathcal{AB}_{g,1}} \ar[r]  & 0.
 	\end{tikzcd}
 	\]
 	
 	But because the $\mathcal{AB}_{g,1}$-equivariant morphism $\lambda$ is defined on $\mathcal{M}_{g,1}(2)$, we get a split exact sequence of rational vector spaces
 	\[
 	\begin{tikzcd}
 		0 \ar[r] & \mathbb{Q}  \ar[r]  & \left(\dfrac{\mathcal{M}_{g,1}(2)}{[\mathcal{M}_{g,1}(2),\mathcal{M}_{g,1}(2)]}\otimes \mathbb{Q}\right)_{\mathcal{AB}_{g,1}} \ar[r, "\tau_2"]& \mathfrak{m}_{g,1}(2)_{\mathcal{AB}_{g,1}} \ar[r]  & 0.
 	\end{tikzcd}
 	\]
 	As the action of $\mathcal{AB}_{g,1}$ on  $\mathfrak{m}_{g,1}(2)$ factors via $GL_g(\mathbb{Z})$ we have the desired result.  
 \end{proof}

 \begin{proof}[Proof of Theorem 3.1] Let $F: \mathcal{S} \rightarrow \mathbb{Q}$ be a normalized invariant whose associated maps $F_g: \mathcal{M}_{g,1}(2) \rightarrow \mathbb{Q}$ are group homomorphisms. Then, by the conjugacy condition, $F_g$ factors via $H_1(\mathcal{M}_{g,1}(2);\mathbb{Q})_{\mathcal{AB}_{g,1}}\simeq \mathbb{Q}\oplus (\mathfrak{m}_{g,1}(2))_{GL_g(\mathbb{Z})}$, and by stability   $\lim_g F_g$ factors through  $\lim_g  H_1(\mathcal{M}_{g,1}(2);\mathbb{Q})_{\mathcal{AB}_{g,1}}$. The $\mathbb{Q}$ factor is induced by the Casson invariant, and the group endomorphisms of $\mathbb{Q}$ are determined by their value at $1$, hence up to adding some multiple of $\lambda$ to $F$ we may assume that for $g \geq 6$, $F_g$ factors through the composition  $\mathcal{M}_{g,1}(2) \stackrel{\tau_2}{\longrightarrow} \mathcal{M}_{g,1}(2)/\mathcal{M}_{g,1}(3) \rightarrow \mathfrak{m}_{g,1}(2) \rightarrow \mathfrak{m}_{g,1}(2)_{GL_g(\mathbb{Z})}$. In \cite{pitsch3} it is shown that $\tau_2(\mathcal{M}_{g,1}(2)) = \tau_2(\mathcal{A}_{g,1}(2)) + \tau_2(\mathcal{B}_{g,1}(2))$, and since an invariant has to vanish on this sum because the involved mapping classes all build the sphere $\mathbb{S}^3$, we conclude that for $g \geq 6$, up to adding some multiple of $\lambda$, $F_g=0$ and hence $F=0$.
\end{proof}

%%%%%%%%%%%%%%%%%%%%%%%%%%%%%%%%%%%%%%%%%%%%%%%%%%%%%%
%%%%%%%%%%%%%%%%%%%%%%%%%%%%%%%%%%%%%%%%%%%%%%%%%%%%%%
\section{Cocycles of invariants of degree at most $2$}\label{sec:cocyandinvdeg2}
%%%%%%%%%%%%%%%%%%%%%%%%%%%%%%%%%%%%%%%%%%%%%%%%%%%%%%
%%%%%%%%%%%%%%%%%%%%%%%%%%%%%%%%%%%%%%%%%%%%%%%%%%%%%%

Let $F: \mathcal{S} \rightarrow \mathbb{Q}$ be an invariant of degree at most $2$, and for each $g \geq 1$, denote by $C_g$ the associated  trivialized $2$-cocycle.
 The key observation for the present work is
%%%%%%%%%%%%%%%%%%%%%%%%%%%%%%%%%%%%%%%%%%%%%%%%%%%%%%%%%%%%
\begin{lemma}\label{lem:type2givesbilincocycle}
%%%%%%%%%%%%%%%%%%%%%%%%%%%%%%%%%%%%%%%%%%%%%%%%%%%%%%%%%%%%
	Let $F: \mathcal{S} \rightarrow  \mathbb{Q}$ be a normalized invariant of finite type $2$. Then, its associated $2$-cocycles on the Johnson subgroup $\mathcal{M}_{g,1}(2)$ are bilinear maps.
\end{lemma}
%%%%%%%%%%%%%%%%%%%%%%%%%%%%%%%%%%%%%%%%%%%%%%%%%%%%%%%%%%%%

\begin{proof}
	By assumption $F$ vanishes on $\mathcal{S}_3$, hence the maps $F_g$ vanish on the augmentation powers $(I\mathcal{T}_{g,1})^5$ and $(I\mathcal{T}_{g,1})^6$. By Johnson's computations of the abelianization of the Torelli group~\cite{jon_3}, $\sqrt{\mathcal{T}_{g,1}(2)} = \mathcal{M}_{g,1}(2)$, hence $I\mathcal{M}_{g,1}(2) \subseteq (I  \mathcal{T}_{g,1})^2$, and $F_g$ viewed as a function on the Johnson subgroup vanishes on $(I\mathcal{M}_{g,1}(2))^3 \subseteq (I\mathcal{T}_{g,1})^6$. Let us show that for a fixed value of $g$, the $2$-cocycle $C_g$ associated to $F_g$ is a homomorphism in the first variable. Let  $a,b,c \in \mathcal{M}_{g,1}(2)$. As  $F_g$ vanishes on the product $(1-a)(1-b)(1-c) \in (I\mathcal{M}_{g,1}(2))^3$, from the equalities
	\begin{align*}
		(1-a)(1-b)(1-c) & = 1-b-c +bc -a +ab +ac -abc \\
		& = 1 -b -c +bc - a -c +ac +ab +c -abc 
	\end{align*}
	and the fact that $F_g(1) = 0$, we deduce that
	\begin{align*}
		0 & = -F_g(b) -F_g(c) +F_g(bc) - F_g(a) - F_g(c) + F_g(ac) + F_g(ab) + F_g(c) - F_g(abc) \\
		& = -C_g(b,c) -C_g(a,c) + C_g(ab,c).
	\end{align*}
	The proof that $C_g$ is a homomorphism in the second variable is analogous.
	
\end{proof}

%%%%%%%%%%%%%%%%%%%%%%%%%%%%%%%%%%%%%%%%%%%%%%%%%%%%%%%%%%%%
\begin{lemma}\label{lem:2cocyclefactors}
%%%%%%%%%%%%%%%%%%%%%%%%%%%%%%%%%%%%%%%%%%%%%%%%%%%%%%%%%%%%
For each $g \geq 1$, the $2$-cocyle $C_g$ associated to a degree $2$ invariant, viewed as a bilinear map on the Johnson subgroup factors through $\mathcal{M}_{g,1}(2)/\sqrt{\mathcal{T}_{g,1}(3)}\otimes \mathbb{Q} \times \mathcal{M}_{g,1}(2)/\sqrt{\mathcal{T}_{g,1}(3)}\otimes  \mathbb{Q}$.
\end{lemma}
%%%%%%%%%%%%%%%%%%%%%%%%%%%%%%%%%%%%%%%%%%%%%%%%%%%%%%%%%%%%

\begin{proof}
Since $\mathcal{M}_{g,1}(2)=\sqrt{\mathcal{T}_{g,1}(2)}$, from Proposition~\ref{prop:coc-finite-type} we know that the 2-cocycle $C_g$ is zero on $\mathcal{M}_{g,1}(2)\times \sqrt{\mathcal{T}_{g,1}(3)},$ $\sqrt{\mathcal{T}_{g,1}(3)}\times \mathcal{M}_{g,1}(2)$. By Lemma~\ref{lem:type2givesbilincocycle}, in each variable it factors then through the quotient $\mathcal{M}_{g,1}(2)/\sqrt{\mathcal{T}_{g,1}(3)}$. In each variable the $2$-cocycle is also a morphism with rational values, hence  it factors through the rationalization $\mathcal{M}_{g,1}(2)/\sqrt{\mathcal{T}_{g,1}(3)}\otimes  \mathbb{Q}$.
\end{proof}

%%%%%%%%%%%%%%%%%%%%%%%%%%%%%%%%%%%%%%%%%%%%%%%%%%%%%%
\begin{proposition}\label{prop:H1KmodT3}
%%%%%%%%%%%%%%%%%%%%%%%%%%%%%%%%%%%%%%%%%%%%%%%%%%%%%%
The Casson invariant $\lambda$ and the second Johnson homomorphism $\tau_2$ induce an isomorphism
\[
\begin{tikzcd}
	\dfrac{\mathcal{M}_{g,1}(2)}{\sqrt{\mathcal{T}_{g,1}(3)}} \otimes \mathbb{Q} \ar[r,"\sim"]& \mathbb{Q}\oplus\mathfrak{m}_{g,1}(2).
\end{tikzcd}
\]
\end{proposition}
%%%%%%%%%%%%%%%%%%%%%%%%%%%%%%%%%%%%%%%%%%%%%%%%%%%%%%

\begin{proof} 
	
By definition we have a short exact sequence:
\[
\begin{tikzcd}
	1 \ar[r] & \mathcal{M}_{g,1}(3)  \ar[r]  & \mathcal{M}_{g,1}(2) \ar[r,"\tau_2"]& \dfrac{\mathcal{M}_{g,1}(2)} {\mathcal{M}_{g,1}(3)} \ar[r]& 0.
\end{tikzcd}
\]
By \cite{faes1} we know that the   Casson invariant $\lambda: \mathcal{M}_{g,1}(3)\rightarrow \mathbb{Z}$ is surjective, and its kernel is $\sqrt{\mathcal{T}_{g,1}(3)}$. Therefore we have a short exact sequence:
\[
\begin{tikzcd}
	1 \ar[r] & \sqrt{\mathcal{T}_{g,1}(3)} \ar[r]  & \mathcal{M}_{g,1}(2) \ar[r,"{(\lambda,\tau_2)}"]& \mathbb{Z}\oplus\dfrac{\mathcal{M}_{g,1}(2)} {\mathcal{M}_{g,1}(3)} \ar[r]  & 0.
\end{tikzcd}
\]

After rationalizing we get an isomorphism
\[
\begin{tikzcd}
	\dfrac{\mathcal{M}_{g,1}(2)}{\sqrt{\mathcal{T}_{g,1}(3)}} \otimes \mathbb{Q} \ar[r,"\sim"]& \mathbb{Q}\oplus\mathfrak{m}_{g,1}(2).
\end{tikzcd}
\]
\end{proof}

As a consequence of Lemma~\ref{lem:2cocyclefactors} and Proposition~\ref{prop:H1KmodT3}, any 2-cocycle associated to an invariant of finite type 2 has to be given by the pull-back of a bilinear form $B_g$ on $\mathbb{Q}\oplus \mathfrak{m}_{g,1}(2)$ along $\lambda\oplus \tau_2.$

A family of bilinear forms $(B_g)_{g\geq 3}$ on $\mathbb{Q}\oplus \mathfrak{m}_{g,1}(2)$ whose lift to $\mathcal{M}_{g,1}(2)$ satisfy properties (1)-(3), given in Section \ref{sec:InvTrivcoc}, has the following properties:
\begin{enumerate}[$(1')$]
	\item The bilinear forms $(B_g)_{g \geq 3}$ are compatible with the stabilization map, in other words, for $g\geq 3$ there is a commutative triangle:
	\[
	\xymatrix{
		(\mathbb{Q}\oplus \mathfrak{m}_{g,1}(2)) \times (\mathbb{Q}\oplus \mathfrak{m}_{g,1}(2)) \ar@{->}[r]   \ar[dr]_-{B_{g}} & (\mathbb{Q}\oplus \mathfrak{m}_{g+1,1}(2)) \times (\mathbb{Q}\oplus \mathfrak{m}_{g+1,1}(2)) \ar[d]^-{B_{g+1}} \\
		& \mathbb{Q}}
	\]
	\item If either $\phi\in \tau_2(\mathcal{A}_{g,1}(2))$ or $\psi \in \tau_2(\mathcal{B}_{g,1}(2))$ then $B_g(\phi, \psi)=0,$
	\item The bilinear forms $(B_g)_{g \geq 3}$ are invariant under conjugation by elements in $GL_g(\mathbb{Z}).$
\end{enumerate}

It turns out that the bilinear map associated to any invariant of type $2$ splits along the decomposition $\mathbb{Q} \oplus \mathfrak{m}_{g,1}(2)$.

%%%%%%%%%%%%%%%%%%%%%%%%%%%%%%%%%%%%%%%%%%%%%%%%%%%%%%%%%%%%%%%%%%%%%%%
\begin{proposition}\label{prop:genm2glinv}
	%%%%%%%%%%%%%%%%%%%%%%%%%%%%%%%%%%%%%%%%%%%%%%%%%%%%%%%%%%%%%%%%%%%%%%%
	For $g\geq 3$ the $\mathbb{Q}$-vector space $\mathfrak{m}_{g,1}(2)_{GL_g(\mathbb{Z})} \simeq  (\mathcal{A}_2(H_\mathbb{Q}))_{GL_g(\mathbb{Z})}$ is generated by elements that belong to $\tau_2(\mathcal{AB}_{g,1}(2))$.
\end{proposition}
%%%%%%%%%%%%%%%%%%%%%%%%%%%%%%%%%%%%%%%%%%%%%%%%%%%%%%%%%%%%%%%%%%%%%%%
\begin{proof}
By Proposition \ref{prop:tree-deriv1}, for $g \geq 3$, $\mathfrak{m}_{g,1}(2)_{GL_g(\mathbb{Z})} \simeq  (\mathcal{A}_2(H_\mathbb{Q}))_{GL_g(\mathbb{Z})}$ is generated by the following trees:
\[
\tree{a_1}{b_1}{b_2}{a_2} \quad \text{and}\quad \tree{a_1}{a_2}{b_2}{b_1}.
\]
By the results of Faes~\cite{faes1}, see Section~\ref{sec:Lagtraces}, these trees belong to $\tau_2(\mathcal{AB}_{g,1}(2))$. 	
\end{proof}

%%%%%%%%%%%%%%%%%%%%%%%%%%%%%%%%%%%%%%%%%%%%%%%%%%%%%%
\begin{theorem}\label{thm:bilinearonKsplits}
%%%%%%%%%%%%%%%%%%%%%%%%%%%%%%%%%%%%%%%%%%%%%%%%%%%%%%
	For a given integer $g\geq 3$, every bilinear form on $\mathbb{Q}\oplus \mathfrak{m}_{g,1}(2)$ that satisfies conditions $(2')$ and $(3')$ is the sum of a bilinear form on $\mathbb{Q}$ and a bilinear form on $\mathfrak{m}_{g,1}(2)$.
\end{theorem}
%%%%%%%%%%%%%%%%%%%%%%%%%%%%%%%%%%%%%%%%%%%%%%%%%%%%%%
\begin{proof}
Let $B_g$ be a bilinear form on $\mathbb{Q}\oplus \mathfrak{m}_{g,1}(2)$ that satisfies conditions $(2')$ and $(3')$. We show that $B_g$ is zero on $\mathbb{Q}\otimes \mathfrak{m}_{g,1}(2)$ and
$\mathfrak{m}_{g,1}(2)\otimes \mathbb{Q}$.
The bilinear form $B_g$ restricted to $\mathbb{Q}\otimes \mathfrak{m}_{g,1}(2)$ is an element of the group:
	\[
	\Hom(\mathbb{Q}\otimes\mathfrak{m}_{g,1}(2),\mathbb{Q})^{GL_g(\mathbb{Z})}=\Hom(\mathbb{Q}\otimes\mathfrak{m}_{g,1}(2)_{GL_g(\mathbb{Z})},\mathbb{Q}),
	\]
where the $GL_g(\mathbb{Z})$-invariants
comes from condition $(3')$.

By Proposition~\ref{prop:genm2glinv},  $\mathfrak{m}_{g,1}(2)_{GL_g(\mathbb{Z})}$ is generated by elements in $\tau_2(\mathcal{AB}_{g,1}(2))$, and by condition $(2')$ the aforementioned map $B_g$ is zero on $\mathbb{Q}\otimes \tau_2(\mathcal{AB}_{g,1}(2))$. Then, by above equality, the bilinear map $B_g$ is zero on $\mathbb{Q}\otimes \mathfrak{m}_{g,1}(2)$.
An analogous argument shows that the bilinear map $B_g$ is also zero on $\mathfrak{m}_{g,1}(2)\otimes \mathbb{Q}$.

\end{proof}

In the sequel, we denote by $C_g$ the $2$-cocycle on $\mathcal{M}_{g,1}(2)$ associated to Ohtsuki's invariant $\lambda_2$. By the above discussion $C_g$ is a sum of a bilinear form $B_g$ on $\mathfrak{m}_{g,1}(2)$ and   a bilinear form on $\mathbb{Q}$. The latter is rather easy to identify: it is given by the product on $\mathbb{Q}$. Since the $\mathbb{Q}$ copy in the abelianization of $\mathcal{M}_{g,1}(2)$ corresponds to the Casson invariant $\lambda$, this bilinear map is given by 
\[
\begin{array}{rcl}
\lambda\cdot\lambda: \mathcal{M}_{g,1}(2) \times \mathcal{M}_{g,1}(2) & \longrightarrow & \mathbb{Q} \\
(\phi,\psi) & \longmapsto & \lambda(\phi)\lambda(\psi).
\end{array}
\] 

Therefore for each $g\geq 3$ there exists a rational number $\alpha_g$ such that $C_g= 2\alpha_g \lambda\cdot\lambda + B_g$. All these rational numbers $\alpha_g$ are in fact equal. In~\cite[Section 2.3]{faes2} Faes constructed an explicit element $\psi \in \mathcal{M}_{2,1}(3)$ such that $\lambda(\mathbb{S}^3_\psi)=1$. By stability we may view this element as belonging to any $\mathcal{M}_{g,1}(3)$ for $g \geq 2$. Then, as by definition $B_{g}(\psi,\psi) =0$,
\[
\begin{aligned}
C_g(\psi,\psi) & =2\alpha_g\lambda(\psi)\lambda(\psi)+B_g(\psi,\psi), \\
C_g(\psi,\psi) & =2\alpha_g.
\end{aligned}
\]
But the left-hand side is independent of $g$ since it is equal to $\lambda_2(\mathbb{S}^3_{\psi^2}) - 2\lambda_2(\mathbb{S}^3_{\psi})$. Therefore $\alpha_g$ does not depend of $g$ and hence there exists a rational number $\alpha$ such that for all $g\geq 3$,
\[
C_g= 2\alpha \lambda\cdot\lambda + B_g.
\]

From Morita~\cite{mor} we also know that the Casson invariant $\lambda$ restricted to $\mathcal{M}_{g,1}(2)$ is a group homomorphism and is an invariant of degree $1$, then its square is an invariant of degree $2$ with associated $2$-cocycle
\[
\lambda^2(\phi\psi)- \lambda^2(\phi)- \lambda^2(\psi)= 2\lambda(\phi)\lambda(\psi).
\]

As a consequence $B_g$ is the $2$-cocycle associated to the invariant $d_2=\lambda_2-\alpha \lambda^2$. As $B_g$ is a bilinear form pulled back from the quotient $\mathfrak{m}_{g,1}(2)$, the  invariant $d_2$ is a group homomorphism when restricted to the third layer of the Johnson filtration $\mathcal{M}_{g,1}(3)$.

%%%%%%%%%%%%%%%%%%%%%%%%%%%%%%%%%%%%%%%%%%%%%%%%%%%%%%
%%%%%%%%%%%%%%%%%%%%%%%%%%%%%%%%%%%%%%%%%%%%%%%%%%%%%%
\section{Cocycles and restriction of invariants to the Johnson  filtration.}
%%%%%%%%%%%%%%%%%%%%%%%%%%%%%%%%%%%%%%%%%%%%%%%%%%%%%%
%%%%%%%%%%%%%%%%%%%%%%%%%%%%%%%%%%%%%%%%%%%%%%%%%%%%%%

In this section we use the knowledge we have on the cocycles for the Casson invariant $\lambda$ and the invariant $\lambda_2$ to understand their behaviour when restricted to the different subgroups in the Johnson filtration.

%%%%%%%%%%%%%%%%%%%%%%%%%%%%%%%%%%%%%%%%%%%%%%%%%%%%%%
\begin{theorem}\label{thm:invrastJohnsfilt}
%%%%%%%%%%%%%%%%%%%%%%%%%%%%%%%%%%%%%%%%%%%%%%%%%%%%%%
	For any $k \geq 1$, we have, for $g$ large enough,
	\[
	\lambda\vert_{\mathcal{M}_{g,1}(k)} \neq 0\quad \text{and} \quad d_{2}\vert_{\mathcal{M}_{g,1}(k)} \neq 0.
	\]
\end{theorem}
%%%%%%%%%%%%%%%%%%%%%%%%%%%%%%%%%%%%%%%%%%%%%%%%%%%%%%

%%%%%%%%%%%%%%%%%%%%%%%%%%%%%%%%%%%%%%%%%%%%%%%%%%%%%%%%%%%%
\begin{rem}
%%%%%%%%%%%%%%%%%%%%%%%%%%%%%%%%%%%%%%%%%%%%%%%%%%%%%%%%%%%%
The statement for the Casson invariant was proved by Hain~\cite{hain1}, we still write down the proof as it is parallel to the proof for the invariant $d_2$.
\end{rem}
%%%%%%%%%%%%%%%%%%%%%%%%%%%%%%%%%%%%%%%%%%%%%%%%%%%%%%%%%%%%

\begin{proof}
	Recall from Proposition \ref{prop:k=1234} that for $k \leq 4$ the construction of integral homology spheres by gluing handlebodies along their boundary induces bijections:
	\[
	\lim_{g \to +\infty} \mathcal{M}_{g,1}(k)/ \sim \rightarrow \mathcal{S}^3,
	\]
where two elements $\phi,\psi \in \mathcal{M}_{g,1}(k)$ are equivalent if and only if there exist maps $\xi_a \in \mathcal{A}_{g,1}$, $\xi_b \in \mathcal{B}_{g,1}$ such that
$\phi = \xi_a \psi \xi_b$, and the limit is taken over the stabilization maps induced on the quotient.
	
\begin{enumerate}
	
\item Proof for the Casson invariant. 
	
	Because any integral homology sphere can be constructed from an element in $\mathcal{M}_{g,1}(4)$, for $g$ large enough $\lambda\vert_{\mathcal{M}_{g,1}(4)} \neq 0$. Assume that for some $k \geq 5$ we have that  $\lambda\vert_{\mathcal{M}_{g,1}(k)} = 0$, and choose $k$ minimal for this property. We know that the Casson invariant is a group homomorphism on $\mathcal{M}_{g,1}(2)$, hence $\lambda\vert_{\mathcal{M}_{g,1}(k-1)}$ is also a group homomorphism and by hypothesis it factors through $\mathcal{M}_{g,1}(k-1)/\mathcal{M}_{g,1}(k)$. As the Casson invariant is defined as a map on the whole Torelli group we can consider its restriction to the lower central series, and as it is of finite type $1$ it vanishes on the terms of the lower central series $\mathcal{T}_{g,1}(\ell)$ for $\ell \geq 2$. Hence we have a commutative diagram between group homomorphisms:
	\[
	\begin{tikzcd}
		\mathcal{T}_{g,1}(k-1)/\mathcal{T}_{g,1}(k) \ar[dr,"\lambda=0"']  \ar[rr] & & \mathcal{M}_{g,1}(k-1)/\mathcal{M}_{g,1}(k) \ar[dl, "\lambda\neq 0"] \\
		& \mathbb{Q} & \\
	\end{tikzcd}
	\]
	
	Tensoring with $\mathbb{Q}$ this commutative diagram of group homomorphism we get
	\[
	\begin{tikzcd}
		\mathfrak{t}_{g,1}(k-1) \ar[dr,"\lambda=0"']  \ar[rr] & & \mathfrak{m}_{g,1}(k-1), \ar[dl, "\lambda\neq 0"] \\
		& \mathbb{Q} & \\
	\end{tikzcd}
	\]
	but by Hain~\cite{hain1}, the horizontal homomorphism is surjective, a contradiction.
	
	\item Proof for the invariant $d_2$.
	
 We proceed similarly. We know that $d_2$ is a group homomorphism when restricted to the groups $\mathcal{M}_{g,1}(3)$. Because $d_2$ is not a trivial invariant, for $g$ large enough $d_2\vert_{\mathcal{M}_{g,1}(4)} \neq 0$. Assume that for some large enough $g$ the invariant $d_2$ vanishes on $\mathcal{M}_{g,1}(k)$ for some $k \geq 5$, and fix $k$ minimal for this vanishing. As before, we have a commutative diagram
 \[
 \begin{tikzcd}
 	\mathcal{T}_{g,1}(k-1)/\mathcal{T}_{g,1}(k) \ar[dr,"d_2=0"']  \ar[rr] & & \mathcal{M}_{g,1}(k-1)/\mathcal{M}_{g,1}(k) \ar[dl, "d_2\neq 0"] \\
 	& \mathbb{Q} & \\
 \end{tikzcd}
 \]
 and its rationalization is
 \[
 \begin{tikzcd}
 	\mathfrak{t}_{g,1}(k-1) \ar[dr,"d_2=0"']  \ar[rr] & & \mathfrak{m}_{g,1}(k-1) \ar[dl, "d_2\neq 0"] \\
 	& \mathbb{Q} & \\
 \end{tikzcd}
 \]
 where the horizontal arrow is surjective by Hain's results, leading to a contradiction.
\end{enumerate}
\end{proof}

%%%%%%%%%%%%%%%%%%%%%%%%%%%%%%%%%%%%%%%%%%%%%%%%%%%%%%%%%%%%
\begin{theorem}\label{thm:filtisnotfull}
%%%%%%%%%%%%%%%%%%%%%%%%%%%%%%%%%%%%%%%%%%%%%%%%%%%%%%%%%%%%
 There exists a rational number $r$ such that the invariant $d_2 - r \lambda$ is not trivial and vanishes on the image of the map
 \[
\begin{array}{rcl}
\displaystyle{\lim_{g \to \infty}} \mathcal{M}_{g,1}(5) & \longrightarrow & \mathcal{S} \\
\phi & \longmapsto & \mathbb{S}^3_\phi.	
\end{array}
\]

\end{theorem}
%%%%%%%%%%%%%%%%%%%%%%%%%%%%%%%%%%%%%%%%%%%%%%%%%%%%%%%%%%%%

 \begin{proof}
 	 Recall that the invariant $d_2=\lambda_2 - \alpha \lambda^2$ is a group homomorphism on $\mathcal{M}_{g,1}(3)$, a priori with values in $\mathbb{Q}$. It is known (cf.\cite{lin}) that both the Casson invariant and the invariant $\lambda_2$ take values in $\mathbb{Z}$. Therefore, if we write $\alpha=p/q$ as irreducible fraction with $q>0$ (setting $p=0$, $q=1$ if $\alpha=0$), the invariant $qd_2 = q \lambda_2 -q\alpha \lambda^2$ is a group homomorphism on $\mathcal{M}_{g,1}(3)$ that takes values in $\mathbb{Z}$.  By Corollary~\ref{cor:ker_d1}, the Casson invariant $\lambda$ induces a homomorphism
 	 \[
 	 \lim_{g \to \infty} \mathcal{M}_{g,1}(5)/\sqrt{\mathcal{T}_{g,1}(5)} \stackrel{\lambda}{\longrightarrow} \mathbb{Z}.
 	 \]
 	 Since $qd_2$ is an invariant of type $2$, the induced homomorphism $qd_2: \mathcal{M}_{g,1}(5) \rightarrow \mathbb{Z}$,  vanishes on $\sqrt{\mathcal{T}_{g,1}(5)}$ and hence induces another homomorphism 
 	 \[
 	 qd_2 : \lim_{g \to \infty} \mathcal{M}_{g,1}(5)/\sqrt{\mathcal{T}_{g,1}(5)} \longrightarrow \mathbb{Z},
 	 \]
 	 which by \cite[Theorem 7.8]{MSS} has to factor through $\lambda$, i.e. there are integers $b_1,b_2$ such that $b_1qd_2 = b_2\lambda$ on $\mathcal{M}_{g,1}(5)$. Then the rational number we are looking for is $r= b_2/(b_1q)$. By construction  $d_2$ is  an invariant of degree $2$ and not $1$, and $\lambda$ is of degree $1$, hence these invariants are not proportional, and $d_2 -r\lambda$ is not a trivial invariant on $\mathcal{S}$.  
 \end{proof}

We now compute the rational numbers $\alpha$ and $r$ that appear in our definitions, this will in particular give an answer to the first half of \cite[Problem 7.1]{MSS}. 
Recall that by definition the rational number $\alpha$ is such that $\lambda_2 - \alpha \lambda^2$ is a group homomorphism on $\mathcal{M}_{g,1}(3)$.   
%%%%%%%%%%%%%%%%%%%%%%%%%%%%%%%%%%%%%%%%%%%%%%%%%%%%%%%%%%%%
\begin{proposition}\label{prop:coefarecomputed}
%%%%%%%%%%%%%%%%%%%%%%%%%%%%%%%%%%%%%%%%%%%%%%%%%%%%%%%%%%%%
The parameters $r$ and $\alpha$ are $\alpha =18$ and $r= -3$. More precisely
\begin{enumerate}
	\item The cocycle $C_g$ associated to $\lambda_2$ is $C_g= 36\lambda\cdot \lambda + B_g$ for some bilinear form $B_g$ on $\mathcal{M}_{g,1}(2)$ that is pulled-back via the second Johnson homomorphism from a bilinear form on $(\mathfrak{m}_{g,1}(2))_{GL_g(\mathbb{Z})}$.
	\item The restriction of $\lambda_2$ to $\mathcal{M}_{g,1}(5)$ is equal to
\[
\lambda_2 = -3 \lambda + 18 \lambda^2.
\]	
In particular the invariant $\lambda_2+3\lambda -18\lambda^2:\mathcal{S} \rightarrow \mathbb{Q}$ is a non-trivial invariant of degree $2$ that vanishes on homology spheres that can be built from an element in some $\mathcal{M}_{g,1}(5)$.
\end{enumerate}
\end{proposition}
%%%%%%%%%%%%%%%%%%%%%%%%%%%%%%%%%%%%%%%%%%%%%%%%%%%%%%%%%%%%

\begin{proof}
	By construction on $\mathcal{M}_{g,1}(5)$, $d_2 = r \lambda$, hence $\lambda_2 = r \lambda +
	 \alpha \lambda^2$.
	 By Section 5 in \cite{lin} we know that if a homology sphere is a connected sum of two homology spheres $M= M_1\#M_2$, we have that
	 \[
	 \begin{aligned}
	 	\lambda(M)= &\lambda(M_1)+\lambda(M_2),\\
	 	\lambda_2(M)= & \lambda_2(M_1)+\lambda_2(M_2)+36\lambda(M_1)\lambda(M_2).
	 \end{aligned}
	 \]
	 By work of Hain (or point $(1)$ in  Theorem~\ref{thm:invrastJohnsfilt}), we know that there exists an element $f\in \mathcal{M}_{g,1}(5)$ such that $\lambda(f)\neq 0$.
	 If we consider $M_1=M_2=\mathbb{S}^3_f$, we have that $M= M_1\#M_2=\mathbb{S}^3_{f \# f}$ with $f \# f \in \mathcal{M}_{2g,1}(5)$ given by viewing $\Sigma_{2g,1}$ as two copies of $\Sigma_{g,1}$ glued along a pair of pants, and applying $f$ to each copy.  Then, 	
	 \[
	 \begin{aligned}
	 	\lambda_2(\mathbb{S}^3_{f \# f})= & \lambda_2(\mathbb{S}^3_f )+\lambda_2(\mathbb{S}^3_f )+36\lambda(\mathbb{S}^3_f )\lambda(\mathbb{S}^3_f )\\
	 	\lambda_2(f \# f)= & 2\lambda_2(f )+36\lambda^2(f )\\
	 	r\lambda(f \# f)+\alpha \lambda^2(f \# f)= & 2\lambda_2(f )+36\lambda^2(f )\\
	 	2r\lambda(f )+4\alpha \lambda^2(f )= & 2r\lambda(f )+2\alpha \lambda^2(f )+36\lambda^2(f )\\
	 	4\alpha \lambda^2(f )= & 2\alpha \lambda^2(f )+36\lambda^2(f )\\
	 	2\alpha= & 36\\
	 	\alpha= & 18.
	 \end{aligned}
	 \]

	We now compute the value of $r$. By \cite[page 304]{lin2} we know that given $M$ an oriented homology $3$-sphere and $\overline{M}$ the same oriented homology $3$-sphere with reversed orientation, we have that
	\[
	\begin{aligned}
	\lambda(\overline{M}) & = -\lambda(M),\\
	\lambda_2(\overline{M}) & = \lambda_2(M)+6\lambda(M).
	\end{aligned}
	\]
	We use again our element $f\in \mathcal{M}_{g,1}(5)$ such that $\lambda(f)\neq 0$.
	For $M=\mathbb{S}^3_f$, we have that
\[
	\begin{aligned}
	\lambda_2(\overline{M}) & = \lambda_2(M)+6\lambda(M) \\
	r\lambda(\overline{M})+18\lambda^2(\overline{M}) & = r\lambda(M)+18\lambda^2(M)+6\lambda(M)\\
	-r\lambda(M)+18\lambda^2(M) & = 18\lambda^2(M)+(r+6)\lambda(M)\\
	-r\lambda(M) & = (r+6)\lambda(M)\\
	-r & = r+6\\
	r & = -3.
	\end{aligned}
	\]

\end{proof}

%%%%%%%%%%%%%%%%%%%%%%%%%%%%%%%%%%%%%%%%%%%%%%%%%%%%%%%%%%%%
\begin{corollary}
%%%%%%%%%%%%%%%%%%%%%%%%%%%%%%%%%%%%%%%%%%%%%%%%%%%%%%%%%%%%
The Poincaré 3-sphere can not be build from a Heegaard splitting with gluing map an element of $\mathcal{M}_{g,1}(5)$.
\end{corollary}
%%%%%%%%%%%%%%%%%%%%%%%%%%%%%%%%%%%%%%%%%%%%%%%%%%%%%%%%%%%%

\begin{proof} The classical normalization of the Casson invariant is that for the Poincaré sphere $\mathbb{P}$, $\lambda(\mathbb{P})=1$, and for instance by the computations in \cite{lin} we know that  $\lambda_2(\mathbb{P})=39$. Therefore $(\lambda_2+3\lambda -18\lambda^2)(\mathbb{P}) = 24\neq 0$.
\end{proof}

%%%%%%%%%%%%%%%%%%%%%%%%%%%%%%%%%%%%%%%%%%%%%%%%%%%%%%%%%%%%%%%%%%%
%%%%%%%%%%%%%%%%%%%%%%%%%%%%%%%%%%%%%%%%%%%%%%%%%%%%%%%%%%%%%%%%%%%
\section{The bilinear form associated to $\lambda_2$}\label{sec:billambda2}
%%%%%%%%%%%%%%%%%%%%%%%%%%%%%%%%%%%%%%%%%%%%%%%%%%%%%%%%%%%%%%%%%%%
%%%%%%%%%%%%%%%%%%%%%%%%%%%%%%%%%%%%%%%%%%%%%%%%%%%%%%%%%%%%%%%%%%%

In this  section we provide an explicit expression of the 2-cocycle associated to the second Ohtsuki invariant $\lambda_2$, i.e. we find the bilinear form $B_g$ that appears in Proposition~\ref{prop:coefarecomputed}.

%%%%%%%%%%%%%%%%%%%%%%%%%%%%%%%%%%%%%%%%%%%%%%%%%%%%%%%%%%%%%%%%%%%
\begin{proposition}\label{prop:gen-m(2)-(1)-(3)}
%%%%%%%%%%%%%%%%%%%%%%%%%%%%%%%%%%%%%%%%%%%%%%%%%%%%%%%%%%%%%%%%%%%
	For $g \geq 5$, a bilinear form $B_g$ on $\mathfrak{m}_{g,1}(2)$ that satisfies properties
\begin{enumerate}
\item[(2')] If either $\phi\in \tau_2(\mathcal{A}_{g,1}(2))$ or $\psi \in \tau_2(\mathcal{B}_{g,1}(2))$ then $B_g(\phi, \psi)=0,$
	\item[(3')] It is invariant under conjugation by elements in $GL_g(\mathbb{Z}).$
	
\end{enumerate}
	 is completely determined by its values on 
	\[
	\tree{b_1}{b_2}{b_3}{b_4}\otimes \tree{a_1}{a_2}{a_3}{a_4}\quad \text{and}\quad \tree{b_2}{b_3}{b_4}{a_2}\otimes \tree{a_1}{a_3}{a_4}{b_1}.
	\]
\end{proposition}
%%%%%%%%%%%%%%%%%%%%%%%%%%%%%%%%%%%%%%%%%%%%%%%%%%%%%%%%%%%%%%%%%%%

%%%%%%%%%%%%%%%%%%%%%%%%%%%%%%%%%%%%%%%%%%%%%%%%%%%%%%%%%%%%%%%%%%%
\begin{proof} Because of property $(3')$ such a bilinear form is the same as a linear map
	\[
	B_g: (\mathcal{A}_2(H_\mathbb{Q}) \otimes \mathcal{A}_2(H_\mathbb{Q}))_{GL_g(\mathbb{Z})} \rightarrow \mathbb{Q}.
	\]
	
	By  Proposition~\ref{prop:tree-deriv1} we have a $GL_g(\mathbb{Z})$-equivariant epimorphism
	\[
	\xymatrix@C=10mm@R=10mm{\otimes^8 H_\mathbb{Q} \ar@{->>}[r] & \dfrac{S^2(\Lambda^2 H_\mathbb{Q}) }{\Lambda^4 H_\mathbb{Q}} \otimes \dfrac{S^2(\Lambda^2 H_\mathbb{Q}) }{\Lambda^4 H_\mathbb{Q}} \simeq \mathcal{A}_2(H_\mathbb{Q}) \otimes \mathcal{A}_2(H_\mathbb{Q})},
	\]
	Then Proposition~\ref{prop:chords} implies that
	$(\mathcal{A}_2(H_\mathbb{Q}) \otimes \mathcal{A}_2(H_\mathbb{Q}))_{GL_g(\mathbb{Z})}$ is generated by elements in which each pair of elements $\{a_i,b_i\}$ for $1 \leq i \leq g$ appears exactly once or does not appear at all, and each of these falls in precisely one summand in the decomposition
	\begin{equation}
		\label{eq:dec-4trees}
		\mathcal{A}_2(H_\mathbb{Q})= W(a^4)\oplus W(a^3b)\oplus W(a^2b^2)\oplus W(ab^3)\oplus W(b^4).
	\end{equation}
Because of condition $(2')$, the bilinear form $B_g$ factors through the quotient 
\[(\mathcal{A}_2(H_\mathbb{Q}) \otimes \mathcal{A}_2(H_\mathbb{Q}))_{GL_g(\mathbb{Z})}
\rightarrow \left(\dfrac{\mathcal{A}_2(H_\mathbb{Q})}{\tau_2(\mathcal{A}_{g,1}(2))} \otimes \dfrac{\mathcal{A}_2(H_\mathbb{Q})}{\tau_2(\mathcal{B}_{g,1}(2))}\right)_{GL_g(\mathbb{Z})}.\]
From Faes \cite{faes1} computations (see the end of Section~\ref{sec:Lagtraces}), 
\[
\dfrac{\mathcal{A}_2(H_\mathbb{Q})}{\tau_2(\mathcal{A}_{g,1}(2))} \otimes \dfrac{\mathcal{A}_2(H_\mathbb{Q})}{\tau_2(\mathcal{B}_{g,1}(2))} = (W(b^4) \oplus W(ab^3)/W_0(ab^3)) \otimes (W(a^4) \oplus W(a^3b)/W_0(a^3b))
\]
	Set $W_1(a^3b)=W(a^3b)/W_0(a^3b)$ and $W_1(ab^3)=W(ab^3)/W_0(ab^3)$.
	By the arguments from the proof of Proposition~\ref{prop:chords}, any imbalance in the number of labels  $a_j$ and $b_j$ in a generator of $\mathcal{A}_2(H_\mathbb{Q}) \otimes \mathcal{A}_2(H_\mathbb{Q})$ yields a trivial element in the coinvariant quotient. For instance any generator in $W_1(ab^3)\otimes W(a^4)$, necessarily presents such an imbalance, and from this we get that the quotient 
	\[
	\left((W(b^4) \oplus W_1(ab^3)) \otimes (W(a^4) \oplus W_1(a^3b))\right)_{GL_g(\mathbb{Z})}
	\] 
	 boils down to
	\[
	(W(b^4) \otimes W(a^4))_{GL_g(\mathbb{Z})}\oplus (W_1(ab^3) \otimes W_1(a^3b))_{GL_g(\mathbb{Z})}.
	\]
	Next we compute a set of generators for each of these two summands. Recall that from the action of the symmetric group $\mathfrak{S}_g \subseteq GL_g(\mathbb{Z})$, we may assume that the only labels $a_i$ or $b_i$ that appear have indices $1 \leq i \leq 4$. 
	
	To respect the constraints of balancing the pairs that appear, and the $AS$ relation, the coinvariant quotient $(W(b^4) \otimes W(a^4))_{GL_g(\mathbb{Z})}$ is generated a priori by
	\[
	\tree{b_1}{b_2}{b_3}{b_4} \otimes \tree{a_1}{a_2}{a_3}{a_4}
	\]
	and the $23$ other tensors that result from tensoring the tree on the left with the trees obtained by permuting  the indices in the tree on the right. The symmetries in the trees, and the AS relation boil down these 24 tensors to just two:
	\[
	\tree{b_1}{b_2}{b_3}{b_4} \otimes \tree{a_1}{a_2}{a_3}{a_4} 	 \qquad\text{and} \qquad \tree{b_1}{b_2}{b_3}{b_4}\otimes \tree{a_1}{a_4}{a_3}{a_2}.
	\]

	We now apply the IHX relation to the tree on the right  in the first tensor 
	\[
	\begin{aligned}
		\tree{b_1}{b_2}{b_3}{b_4}\otimes \tree{a_1}{a_2}{a_3}{a_4} & =
		\tree{b_1}{b_2}{b_3}{b_4}\otimes \tree{a_2}{a_3}{a_4}{a_1} -
		\tree{b_2}{b_4}{b_3}{b_1}\otimes \tree{a_2}{a_4}{a_3}{a_1} \\
		& =
		\tree{b_1}{b_2}{b_3}{b_4}\otimes \tree{a_1}{a_4}{a_3}{a_2}+
		\tree{b_1}{b_2}{b_4}{b_3}\otimes \tree{a_1}{a_3}{a_4}{a_2},
	\end{aligned}
\]
where the second term of the sum is zero in the coinvariant quotient since the action of the transposition $(34)$ sends the second term of the sum to its opposite, by the AS relation.
Hence $(W(b^4) \otimes W(a^4))_{GL_g(\mathbb{Z})}$ is generated by
\[
\tree{b_1}{b_2}{b_3}{b_4} \otimes \tree{a_1}{a_2}{a_3}{a_4}.
\]

We now compute a set of generators of $(W_1(ab^3) \otimes W_1(a^3b))_{GL_g(\mathbb{Z})}$.
Remember that $W_1(ab^3)$ is the quotient of $W(ab^3)$ by the kernel of the Lagrangian trace map $\Tr^A$.
The module $W(ab^3)$ is generated by $\mathfrak{S}_g$-orbits of the trees:
\[
\tree{b_1}{b_2}{b_3}{a_4}, \quad
\tree{b_1}{b_2}{b_2}{a_4}, \quad 
\tree{b_1}{b_2}{b_4}{a_4}, \quad
\tree{b_4}{b_2}{b_4}{a_4}, \quad 
\tree{b_1}{b_2}{b_4}{a_1} \quad \text{and} \quad 
\tree{b_1}{b_2}{b_2}{a_1} 
.
\]
A direct computation shows that the first three trees and the difference between the fourth and fifth tree belong to the kernel of $\Tr^A$. As a consequence, the module $W_1(ab^3)$ is generated by the $\mathfrak{S}_g$-orbits of last two trees. Moreover, taking the action by the element $G\in GL_g(\mathbb{Z})$ that sends $b_4$ to $b_4+b_2$ and $a_4$ to $a_4-a_2$ we have that
\[
G\Bigg(\tree{b_1}{b_2}{b_4}{a_1}\Bigg)=\tree{b_1}{b_2}{b_4}{a_1}+
\tree{b_1}{b_2}{b_2}{a_1}.
\]
Thus, as a $GL_g(\mathbb{Z})$-module, $W_1(ab^3)$ is generated by the tree $\tree{b_1}{b_2}{b_4}{a_1}$ and the same statement holds true for $W_1(a^3b)$ exchanging $a$'s and $b$'s.

Therefore, to respect the constraints of balancing the pairs that appear, and the $AS$ relation, the coinvariant quotient $(W_1(ab^3) \otimes W_1(a^3b))_{GL_g(\mathbb{Z})}$ is generated a priori by
	\begin{equation}
		\label{eq:gen-bil-a3b}
		\tree{b_2}{b_3}{b_4}{a_2}\otimes \tree{a_1}{a_3}{a_4}{b_1}\qquad \text{and}\qquad \tree{b_2}{b_4}{b_3}{a_2}\otimes \tree{a_1}{a_3}{a_4}{b_1}.
	\end{equation}
Observe that because on the trees on the left side of the tensor $\Tr^A$ takes the non-trivial values $b_3b_4 \in S^2(B)$ and on the trees to the right side of the tensor $\Tr^B$ takes the non-trivial values $-a_3a_4 \in S^2(A)$, these two tensors define indeed non-trivial elements in  $W_1(ab^3) \otimes W_1(a^3b)$.
	By the IHX relation we have that
	\[
	\tree{a_2}{b_2}{b_3}{b_4}=\tree{b_2}{b_3}{b_4}{a_2}-\tree{b_2}{b_4}{b_3}{a_2}.
	\]
	Now $\Tr^A\Big(\tree{a_2}{b_2}{b_3}{b_4}\Big)=0$, hence  this tree belongs to $W_0(ab^3)$, and therefore in the quotient
	\[
	\tree{b_2}{b_3}{b_4}{a_2} \otimes \tree{a_1}{a_3}{a_4}{b_1}= \tree{b_2}{b_4}{b_3}{a_2}\otimes \tree{a_1}{a_3}{a_4}{b_1}.
	\]
\end{proof}
%%%%%%%%%%%%%%%%%%%%%%%%%%%%%%%%%%%%%%%%%%%%%%%%%%%%%%%%%%%%%%%%%%%

We will now introduce two bilinear forms on $\mathfrak{m}_{g,1}(2)$ that satisfy properties $(2'),(3')$ and hence are stable.
On $H= A \oplus B$, with our preferred basis fixed, we have two perfect pairings, on the one hand we have the symplectic intersection form $\omega$, and on the other hand a symmetrical pairing between $A$ and $B$, with matrix in our preferred basis given by
\[
\overline{\omega} = \left(\begin{matrix}
	0 & Id \\
	Id & 0
\end{matrix}\right).
\]
The symplectic form is tautologically $\Sp_{2g}(\mathbb{Z})$-invariant, but the bilinear form $\overline{\omega}$ is only $GL_g(\mathbb{Z}) \subseteq \Sp_{2g}(\mathbb{Z})$-invariant. 

First, we use the symmetric pairing on $H$ to identify $H_\mathbb{Q}$ with its dual. This induces a contraction map:
\[
\begin{array}{rcl}		
	S^2(\Lambda^2H_ \mathbb{Q}) & \longrightarrow & S^2(H_\mathbb{Q}) \\
	(a \wedge b)(c \wedge d) & \longmapsto & \begin{aligned}[t]
		& \overline{\omega}(a,d)bc-\overline{\omega}(a,c)b d \\[-1ex]
		&-\overline{\omega}(b,d)ac+ \overline{\omega}(b,c)ad.
	\end{aligned}
\end{array}
\]

An immediate computation shows that this map vanishes on the subspace $\Lambda^4 H_\mathbb{Q}$ and hence descends to a map:
\[
\begin{array}{rcl}
	C_{S}: \mathcal{A}_2(H_\mathbb{Q}) & \longrightarrow & S^2(H_\mathbb{Q}) \\
	\tree{a}{b}{c}{d}  & \longmapsto  &
	\begin{aligned}[t]
		& \overline{\omega}(a,d)bc-\overline{\omega}(a,c)b d \\[-1ex]
		&-\overline{\omega}(b,d)ac+ \overline{\omega}(b,c)ad.
	\end{aligned}
\end{array}
\]
This map coincides with the Morita trace map defined in \cite[Section 6]{mor-ab} (for an explicit description see \cite[Section 1]{faes3}), but replacing the pairing $\omega$ with $\overline{\omega}$.
Then, from the identification of $H_\mathbb{Q}$ with its dual via the symplectic form, and the fact that the symmetric square is compatible with passing to the dual, we have a canonical perfect pairing
\[
\begin{array}{rcl}
	\eta_{S}: S^2(H_\mathbb{Q})\otimes S^2(H_\mathbb{Q}) & \longrightarrow &\mathbb{Q} \\
	(a b) \otimes (c d) & \longmapsto  & \omega(a,c)\omega(b,d)+\omega(a,d)\omega(b,c).
\end{array}
\]

By composition this gives a first bilinear and $GL_g(\mathbb{Z})$-equivariant map
\[
\Upsilon_g:\mathcal{A}_2(H_\mathbb{Q}) \otimes \mathcal{A}_2(H_\mathbb{Q})\xrightarrow[]{C_S\otimes C_S}
S^2(H_\mathbb{Q})\otimes S^2(H_\mathbb{Q})\xrightarrow[]{\eta_S} \mathbb{Q}.
\]

For our second map, we take advantage from the fact that both the symmetric power and exterior product constructions are compatible with passing to the duals. Therefore, if we denote by $V^\vee$ the dual of the $\mathbb{Q}$-vector space $V$, we have a canonical isomorphism
\[
S^2(\Lambda^2 V^\vee) \simeq S^2(\Lambda^2 V)^\vee
\]
and hence a perfect pairing
\[
S^2(\Lambda^2 V^\vee )\otimes S^2(\Lambda^2 V) \rightarrow \mathbb{Q}.
\]
The symplectic form $\omega$, induces an isomorphism $H_\mathbb{Q} \simeq H_\mathbb{Q}^\vee$, hence in our case the above construction gives us an inner product
\[
S^2(\Lambda^2 H_\mathbb{Q}) \otimes S^2(\Lambda^2 H_\mathbb{Q}) \rightarrow  \mathbb{Q}.
\]
By direct inspection, this inner product restricts to the canonical inner product on $\Lambda^4 H_\mathbb{Q} \subseteq S^2(\Lambda^2 H_\mathbb{Q})$, where the vector space $\Lambda^4 H_\mathbb{Q}$ is embedded in $S^2(\Lambda^2 H_\mathbb{Q})$ by sending $a\wedge b\wedge c \wedge d$ to
$(a\wedge b)(c \wedge d)-(a\wedge c)(b \wedge d)+(a\wedge d)(b \wedge c)$.
By taking orthogonal complements we therefore have an identification
\[
(\Lambda^4 H_\mathbb{Q})^\bot \simeq \mathcal{A}_2(H_\mathbb{Q})
\]
and an induced inner product
\[
\nabla_g: \mathcal{A}_2(H_\mathbb{Q}) \otimes \mathcal{A}_2(H_\mathbb{Q}) \rightarrow \mathbb{Q}.
\]

If we unfold the definitions, an explicit expression of this map is given as follows. Consider the subgroups of $\mathfrak{S}_4$,
$V_4$, the $2$-Sylow subgroup of $\mathfrak{A}_4$, which is generated for instance by $ (1,2)(3,4)$ and $(2,3)(1,4)$, and the cyclic $2$-group $C_2$ generated by $(3,4)$, then
\[\nabla_g: \mathcal{A}_2(H_\mathbb{Q}) \otimes \mathcal{A}_2(H_\mathbb{Q}) \rightarrow \mathbb{Q}\]
\[
\tree{x_1}{x_2}{x_3}{x_4}\otimes \tree{y_1}{y_2}{y_3}{y_4}\;\mapsto
\]
\[
\sum_{\sigma\in V_4, \; \tau\in C_2}
\begin{aligned}[t] & \;\text{Sgn}(\tau)\omega(x_{1},y_{\sigma(1)})\omega(x_{2},y_{\sigma(2)})\omega(x_{\tau(3)},y_{\sigma(3)})\omega(x_{\tau(4)},y_{\sigma(4)}) \\
	&-\dfrac{1}{2}\text{Sgn}(\tau)\omega(x_{1},y_{\sigma(1)})\omega(x_{2},y_{\sigma(3)})\omega(x_{\tau(3)},y_{\sigma(4)})\omega(x_{\tau(4)},y_{\sigma(2)}) \\
	&+\dfrac{1}{2}\text{Sgn}(\tau)\omega(x_{1},y_{\sigma(1)})\omega(x_{2},y_{\sigma(4)})\omega(x_{\tau(3)},y_{\sigma(3)})\omega(x_{\tau(4)},y_{\sigma(2)}),
\end{aligned}
\]
where $\text{Sgn}$ stands for the signature of a permutation.
\begin{remark}
Both maps have a diagrammatic interpretation in terms of Jacobi diagrams (see for instance \cite[Chap. 6]{ohtsuki3}). 

For the first map $\Upsilon_g$, given two $H$-shaped trees, in each tree first glue one of the leaves of the first leg to a leave on the second leg in all possible ways using the maps $\overline{\omega}$, this turns each $H$-shaped tree into a sum of diagrams of the form $\begin{tikzpicture}[baseline={([yshift=-.5ex]current bounding box.center)},scale=.25]
	\draw[thick] (-2,0) circle [radius=1];
	\draw[thick] (-4,0) -- (-3,0);
	\draw[thick] (-1,0) -- (0,0);
	\node[right] at (0,0) { \tiny $d$};
	\node[left] at (-4,0) { \tiny $c$};
\end{tikzpicture}$. Then, using $\omega$, glue the leaves of each diagram coming from the first tree to the leaves of each diagram coming from the second tree the two possible ways to do it. This yields $\Upsilon_g$-times a Jacobi diagram of the shape  $\begin{tikzpicture}[baseline={([yshift=-.5ex]current bounding box.center)},scale=.25]
\draw[thick] (-2,0) circle [radius=1];
\draw[thick] (-2.9,0.5) -- (-1.1,0.5);
\draw[thick] (-2.9,-0.5) -- (-1.1,-0.5);
\end{tikzpicture}$.

For the second map $\nabla_g$, given two $H$-shaped trees, glue the leaves of the left hand tree to the leaves of the right-hand in all possible ways tree using $\omega$. This produces a linear combination of the two Jacobi diagrams
$\begin{tikzpicture}[baseline={([yshift=-.5ex]current bounding box.center)},scale=.25]
	\draw[thick] (-2,0) circle [radius=1];
	\draw[thick] (-2.9,0.5) -- (-1.1,0.5);
	\draw[thick] (-2.9,-0.5) -- (-1.1,-0.5);
\end{tikzpicture}$
and
$\begin{tikzpicture}[baseline={([yshift=-.5ex]current bounding box.center)},scale=.25]
	\draw[thick] (-2,0) circle [radius=1];
	\draw[thick] (-2.9,-0.5) -- (-2,0);
	\draw[thick] (-2,0) -- (-1.1,-0.5);
	\draw[thick] (-2,0) -- (-2,1);
\end{tikzpicture}$.
By IHX relation in the space of Jacobi diagrams
$
\begin{tikzpicture}[baseline={([yshift=-.5ex]current bounding box.center)},scale=.25]
	\draw[thick] (-2,0) circle [radius=1];
	\draw[thick] (-2.9,0.5) -- (-1.1,0.5);
	\draw[thick] (-2.9,-0.5) -- (-1.1,-0.5);
\end{tikzpicture}
= 2\;\;
\begin{tikzpicture}[baseline={([yshift=-.5ex]current bounding box.center)},scale=.25]
	\draw[thick] (-2,0) circle [radius=1];
	\draw[thick] (-2.9,-0.5) -- (-2,0);
	\draw[thick] (-2,0) -- (-1.1,-0.5);
	\draw[thick] (-2,0) -- (-2,1);
\end{tikzpicture},
$
and we send $\begin{tikzpicture}[baseline={([yshift=-.5ex]current bounding box.center)},scale=.25]
	\draw[thick] (-2,0) circle [radius=1];
	\draw[thick] (-2.9,0.5) -- (-1.1,0.5);
	\draw[thick] (-2.9,-0.5) -- (-1.1,-0.5);
\end{tikzpicture}$
to $1$. This explains why we have a factor $\frac{1}{2}$ in the explicit formula for $\nabla_g$.
\end{remark}

By construction the bilinear forms $\Upsilon_g$ and $\nabla_g$ satisfy properties $(1')$ (stability) and $(3')$ ($GL_g(\mathbb{Z})$-invariance). But they do not satisfy property $(2')$ because their diagrammatic interpretation shows that they are symmetric and, since $Im(\tau_2)=\tau_2(\mathcal{A}_{g,1}(2))+\tau_2(\mathcal{B}_{g,1}(2))$ by \cite[Theorem 5.1]{pitsch3},
there are no symmetric bilinear forms that satisfy property $(2')$.
To get bilinear forms that satisfy also property $(2')$ we desymmetrize the aforementioned bilinear forms by restricting them to $W(ab^3)\otimes W(a^3b)$ and $W(b^4)\otimes W(a^4)$ respectively. More precisely, given the decomposition 
\begin{equation*}
	\mathcal{A}_2(H_\mathbb{Q})= W(a^4)\oplus W(a^3b)\oplus W(a^2b^2)\oplus W(ab^3)\oplus W(b^4),
\end{equation*}
we denote by $\pi_{st}: \mathcal{A}_2(H_\mathbb{Q}) \rightarrow W(a^sb^t)$ the canonical projection and set
\[
\begin{tikzcd}[row sep=small]
	Q_g: \mathcal{A}_2(H_\mathbb{Q}) \otimes \mathcal{A}_2(H_\mathbb{Q}) \ar[r,"\pi_{13}\otimes\pi_{31}"] & W(ab^3)\otimes W(a^3b) \ar[r, "\Upsilon_g"] & \mathbb{Q}, \\
	J_g: \mathcal{A}_2(H_\mathbb{Q}) \otimes \mathcal{A}_2(H_\mathbb{Q}) \ar[r,"\pi_{04}\otimes\pi_{40}"] & W(b^4)\otimes W(a^4) \ar[r, "\nabla_g"] & \mathbb{Q}.
\end{tikzcd}
\]

A direct inspection shows that $Q_g$ and $J_g$ satisfy properties $(1')$-$(3')$. Moreover by direct computation, on the generators given in Proposition \ref{prop:gen-m(2)-(1)-(3)} they take the values $(1,0)$ and $(0,1)$ respectively. Therefore we have the following result.

%%%%%%%%%%%%%%%%%%%%%%%%%%%%%%%%%%%%%%%%%%%%%%%%%%%%%%%%%%%%%%%%%%%
\begin{proposition}
\label{prop:bil-gen-1-3}
%%%%%%%%%%%%%%%%%%%%%%%%%%%%%%%%%%%%%%%%%%%%%%%%%%%%%%%%%%%%%%%%%%%
The $\mathbb{Q}$-vector space of bilinear forms on $\mathfrak{m}_{g,1}(2)$ that satisfy properties $(1')$-$(3')$ has dimension $2$ with a basis given by $Q_g$ and $J_g$.\footnote{The map $J_g$ is analogous to the one denoted ${}^tJ_g$ and related to the Casson invariant in~\cite{pitsch}, we drop here the transpose sign which is completely irrelevant.}
\end{proposition}
%%%%%%%%%%%%%%%%%%%%%%%%%%%%%%%%%%%%%%%%%%%%%%%%%%%%%%%%%%%%%%%%%%%

Next we compute the linear combination of $Q_g$ and $J_g$ that gives us the bilinear form $B_g$ that appears in Proposition~\ref{prop:coefarecomputed} and hence the bilinear form $C_g$ associated to $\lambda_2$.
All our computations will be based on the Lickorish trick (c.f. \cite[page 11]{M}), which originally appeared in \cite{lick}, and a formula of Lin and Wang \cite{lin2} to compute the invariants $\lambda$ and $\lambda_2$ of a homology 3-sphere obtained by surgery on a knot in $\mathbb{S}^3$.
To be more precise, we use the following two results:

%%%%%%%%%%%%%%%%%%%%%%%%%%%%%%%%%%%%%%%%%%%%%%%%%%%%%%%%%%%%%%%%%%%
\begin{proposition}[Lickorish's trick]
	\label{prop:trick}
%%%%%%%%%%%%%%%%%%%%%%%%%%%%%%%%%%%%%%%%%%%%%%%%%%%%%%%%%%%%%%%%%%%
Let $U$ and $V$ be two compact $3$-manifolds whose boundaries are identified.
Let $\gamma\subset \partial V$ be a simple closed curve, and let $K\subset \text{int}(V)$ be the knot obtained by slightly ``pushing'' $\gamma$ inside $int(V)$. Then we have a homeomorphism
\[
U\cup_\tau (-V)\simeq U \cup_{\text{id}} (-V_K),
\]	 
where $\tau:=T_\gamma$ is the Dehn twist along $\gamma$, and $V_K$ is obtained from $V$ by surgery along $K$ framed with the parallel differing from $\gamma$ by a meridian of $K$.	 
\end{proposition}
%%%%%%%%%%%%%%%%%%%%%%%%%%%%%%%%%%%%%%%%%%%%%%%%%%%%%%%%%%%%%%%%%%%

%%%%%%%%%%%%%%%%%%%%%%%%%%%%%%%%%%%%%%%%%%%%%%%%%%%%%%%%%%%%%%%%%%%
\begin{proposition}[Theorem 5.2 in \cite{lin2}]
	\label{prop:comput-lambdas}
	%%%%%%%%%%%%%%%%%%%%%%%%%%%%%%%%%%%%%%%%%%%%%%%%%%%%%%%%%%%%%%%%%%%
	Let $n$ be an integer, and $S^3_{K,1/n}$ be the homology 3-sphere obtained from $1/n$-surgery on a knot $K$. Then
	\begin{enumerate}[i)]
		\item $\lambda(S^3_{K,1/n})=-\dfrac{n}{6}\cdot v_2(K)$,
		\item $\lambda_2(S^3_{K,1/n})=\dfrac{n}{2} v_2(K)-\dfrac{n}{3}v_3(K)+n^2 \left[ v_2(K)+\dfrac{5}{3}v_2^2(K)-60c_4(K)\right]$.
	\end{enumerate}
	where $c_4(K)$ is the coefficient of $z^4$ in the Conway polynomial of $K$ and $v_i(K)$ is the $i$-th derivative of the Jones polynomial $V(K,e^{-h})$ at $h=0$.
\end{proposition}
%%%%%%%%%%%%%%%%%%%%%%%%%%%%%%%%%%%%%%%%%%%%%%%%%%%%%%%%%%%%%%%%%%%

%%%%%%%%%%%%%%%%%%%%%%%%%%%%%%%%%%%%%%%%%%%%%%%%%%%%%%%%%%%%%%%%%%%
\begin{proposition}
\label{prop:coefbilarecomputed}
%%%%%%%%%%%%%%%%%%%%%%%%%%%%%%%%%%%%%%%%%%%%%%%%%%%%%%%%%%%%%%%%%%%
The 2-cocycle $C_g$ associated to the second Ohtsuki invariant $\lambda_2$ is the bilinear form
\[ 36\lambda\cdot \lambda +48\tau_2^*(J_g)+ 12\tau_2^* (Q_g).\]
\end{proposition}
%%%%%%%%%%%%%%%%%%%%%%%%%%%%%%%%%%%%%%%%%%%%%%%%%%%%%%%%%%%%%%%%%%%

%%%%%%%%%%%%%%%%%%%%%%%%%%%%%%%%%%%%%%%%%%%%%%%%%%%%%%%%%%%%%%%%%%%
\begin{proof}
By Proposition \ref{prop:bil-gen-1-3} there exist $r_1,r_2\in \mathbb{Q}$ such that
\begin{equation}
\label{eq:bil-lindep}
B_g= r_1\tau_2^*(J_g)+ r_2\tau_2^* (Q_g).
\end{equation}

By construction (Proposition~\ref{prop:coefarecomputed}) we know that given $f,h\in \mathcal{M}_{g,1}(2),$
\begin{equation}
\label{eq:bil-calcul-d2}
B_g (f,h)=\lambda_2(fh)-\lambda_2(f)-\lambda_2(h)-36\lambda(f)\lambda(h).
\end{equation}

Let $K$ denote the right hand trefoil knot and $L$ the figure eight knot embedded in $\Sigma_{2,1}$ as in Figure~\ref{fig:trefoil-eight-surgery}. The curves $K$ and $L$ are bounding simple closed curves of genus $1$, therefore the associated right hand Dehn twists $T_K, T_L$ belong to $\mathcal{M}_{2,1}(2)$. We now compute the values of $\tau_2^*(J_g)$, $\tau_2^* (Q_g)$ and $B_g$ on $T_K\otimes T_K$ and $T_L\otimes T_L$. Then equation \eqref{eq:bil-lindep} will provide us a system of two linear equations in the two variables $r_1$ and $r_2$.

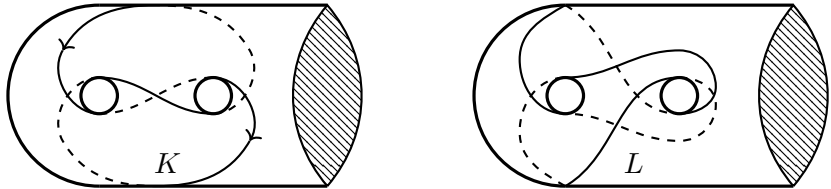
\begin{figure}[H]

	\begin{subfigure}{0.4\textwidth}
	\begin{center}
\begin{tikzpicture}[scale=.6]
\draw[very thick] (-4.5,-2) -- (0.5,-2);
\draw[very thick] (-4.5,2) -- (0.5,2);
\draw[very thick] (-4.5,2) arc [radius=2, start angle=90, end angle=270];
\draw[very thick] (-4.5,0) circle [radius=.4];
\draw[very thick] (-2,0) circle [radius=.4];

\draw[thick] (-4.5,0.4) to [out=0,in=180] (-2,-0.4) ;
\draw[dashed, thick] (-4.5,-0.4) to [out=0,in=180] (-2,0.4);
\draw[dashed, thick] (-2,-0.4) to [out=0,in=-60] (-1.2,1) to [out=120,in=0] (-3.25,2);
\draw[dashed, thick] (-4.5,0.4) to [out=180,in=120] (-5.3,-1) to [out=-60,in=180] (-3.25,-2);
\draw[thick] (-2,0.4) to [out=0,in=60] (-1.2,-1);
\draw[thick] (-1.2,-1) to [out=240,in=0] (-3.25,-2);
\draw[thick] (-4.5,-0.4) to [out=180,in=240] (-5.3,1);
\draw[thick](-3.25,2) to [out=180,in=60] (-5.3,1);

\draw[thick,pattern=north west lines] (0.5,-2) to [out=130,in=-130] (0.5,2) to [out=-50,in=50] (0.5,-2);

\node at (-3,-1.5) {$K$};
\end{tikzpicture}
\end{center}
	%\caption{Caption1}
	%\label{fig:subim1}
	\end{subfigure}
	\begin{subfigure}{0.4\textwidth}
	\begin{center}
\begin{tikzpicture}[scale=.6]
\draw[very thick] (-4.5,-2) -- (0.5,-2);
\draw[very thick] (-4.5,2) -- (0.5,2);
\draw[very thick] (-4.5,2) arc [radius=2, start angle=90, end angle=270];
\draw[very thick] (-4.5,0) circle [radius=.4];
\draw[very thick] (-2,0) circle [radius=.4];

\draw[thick] (-2,-0.4) to [out=0,in=-90] (-1.2,0.2)
to [out=90,in=0] (-2,1) to [out=180,in=0] (-4.5,0.4);
\draw[dashed,thick] (-2,0.4) to [out=0,in=90] (-1.2,-0.2)
to [out=-90,in=0] (-2,-1) to [out=180,in=0] (-4.5,-0.4);

\draw[ thick] (-4.5,-2) to [out=30,in=180] (-2,0.4);
\draw[dashed, thick] (-4.5,2) to [out=-30,in=180] (-2,-0.4);

\draw[ dashed,thick] (-4.5,-2) to [out=150,in=-90] (-5.5,-0.8) to [out=90,in=180] (-4.5,0.4);
\draw[ thick] (-4.5,2) to [out=-150,in=90] (-5.5,0.8) to [out=-90,in=180] (-4.5,-0.4);

\draw[thick,pattern=north west lines] (0.5,-2) to [out=130,in=-130] (0.5,2) to [out=-50,in=50] (0.5,-2);

\node at (-3,-1.5) {$L$};

\end{tikzpicture}
\end{center}
%\caption{Caption2}
%\label{fig:subim2}
\end{subfigure}

	\caption{Trefoil knot $K$ and eight knot $L$ embedded in $\Sigma_{2,1}$}
	\label{fig:trefoil-eight-surgery}
\end{figure}

\textbf{For the right hand trefoil knot $K$.}
The Conway polynomial and the Jones polynomial of $K$ are respectively (cf. \cite[Section 1.2]{ohtsuki3})
\[\nabla(z)=z^2+1 \quad\text{and}\quad V(K,t)=t+t^3-t^4.\]
From this we get $c_4(K)=0$ and $v_2(K)=-6$, hence by Propositions \ref{prop:trick} and \ref{prop:comput-lambdas} we have that $\lambda(T_K)=\lambda(S^3_{K,1})=1$ and 
\[
\lambda_2(T_K^2)-2\lambda_2(T_K)=\lambda_2(S^3_{K,1/2})-2\lambda_2(S^3_{K,1})=2(v_2(K)+\dfrac{5}{3}v_2^2(K)-60c_4(K))=108.
\]
Therefore by equation \eqref{eq:bil-calcul-d2} we get that
\begin{equation}
\label{eq:lincomb-bil}
B_g(T_K,T_K) =72.
\end{equation}

Next we compute $\tau_2(T_K)$.
Consider the colored curves in Figure~\ref{fig:trefoil-surgery}. These two curves form a symplectic basis of the first homology group of the subsurface bounded by $K$ that does not contain the marked disk.

\begin{figure}[H]
	\begin{center}
\begin{tikzpicture}[scale=.6]
\draw[very thick] (-4.5,-2) -- (0.5,-2);
\draw[very thick] (-4.5,2) -- (0.5,2);
\draw[very thick] (-4.5,2) arc [radius=2, start angle=90, end angle=270];
\draw[very thick] (-4.5,0) circle [radius=.4];
\draw[very thick] (-2,0) circle [radius=.4];

\draw[thick] (-4.5,0.4) to [out=0,in=180] (-2,-0.4) ;
\draw[dashed, thick] (-4.5,-0.4) to [out=0,in=180] (-2,0.4);
\draw[dashed, thick] (-2,-0.4) to [out=0,in=-60] (-1.2,1) to [out=120,in=0] (-3.25,2);
\draw[dashed, thick] (-4.5,0.4) to [out=180,in=120] (-5.3,-1) to [out=-60,in=180] (-3.25,-2);
\draw[thick] (-2,0.4) to [out=0,in=60] (-1.2,-1);
\draw[thick] (-1.2,-1) to [out=240,in=0] (-3.25,-2);
\draw[thick] (-4.5,-0.4) to [out=180,in=240] (-5.3,1);
\draw[thick](-3.25,2) to [out=180,in=60] (-5.3,1);

\draw[dashed, very thick, blue] (-4.2,0.2) to [out=40,in=0] (-4.5,1) to [out=180,in=60] (-6.5,0);
\draw[very thick, blue] (-4.2,0.2) to [out=-20,in=90] (-3.6,-0.6) to [out=-90,in=-60] (-6.5,0);

\draw[very thick, red] (-1.6,0) to [out=-20,in=90] (-1.3,-0.6) to [out=-90,in=0] (-2,-1) to [out=180,in=-20] (-4.2,-0.2);
\draw[dashed, very thick, red] (-1.6,0) to [out=20,in=-90] (-1.3,0.6) to [out=90,in=0] (-2,1) to [out=180,in=20] (-4.2,-0.2);

\draw[thick,pattern=north west lines] (0.5,-2) to [out=130,in=-130] (0.5,2) to [out=-50,in=50] (0.5,-2);
\end{tikzpicture} 
\end{center}
	\caption{Trefoil knot embedded in $\Sigma_{2,1}$}
	\label{fig:trefoil-surgery}
\end{figure}
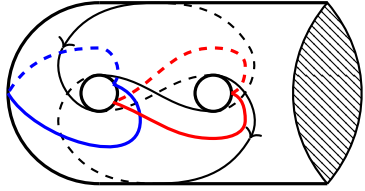

Moreover the homology classes of blue and red curves are $a_1+b_1$ and $a_2-b_1+b_2$ respectively. Then by \cite[Lemma 2.3]{faes1} we have that

\[\tau_2(T_{K})=\frac{1}{2}\tree{a_1+b_1\quad\quad}{a_2-b_1+b_2\quad\quad}{\quad\quad a_1+b_1}{\quad\quad a_2-b_1+b_2}\]

The projections of $\tau_2(T_K)$ to $W(b^4)$, $W(ab^3)$, $W(a^4)$ and $W(a^3b)$ are then
\[
\begin{aligned}
\pi_{04}(\tau_2(T_K)) & = \frac{1}{2}\tree{b_1}{b_2}{b_1}{b_2}, \\
\pi_{40}(\tau_2(T_K)) & = \frac{1}{2}\tree{a_1}{a_2}{a_1}{a_2},
\end{aligned}
\qquad
\begin{aligned}
\pi_{13}(\tau_2(T_K)) & =\tree{a_1}{-b_1+b_2}{b_1}{b_2}+\tree{b_1}{a_2}{b_1}{b_2}, 
\\
\pi_{31}(\tau_2(T_K)) & = \tree{b_1}{a_2}{a_1}{a_2}+\tree{a_1}{-b_1+b_2}{a_1}{a_2} .
\end{aligned}
\]

%%%%%%%%%%%%%%%%%%%%%%%%%%%%%%%%%%%%%%%%%%%%%%%%%%%%%%%%%%%%%%%%%%%%%%%%%%%%%%%
\begin{comment}
\[
\begin{aligned}
\pi_{13}(\tau_2(T_K)) & = \tree{a_1}{-b_1+b_2}{b_1}{-b_1+b_2}+\tree{b_1}{a_2}{b_1}{-b_1+b_2} =\tree{a_1}{-b_1+b_2}{b_1}{b_2}+\tree{b_1}{a_2}{b_1}{b_2}, 
\\
\pi_{31}(\tau_2(T_K)) & = \tree{b_1}{a_2}{a_1}{a_2}+\tree{a_1}{-b_1+b_2}{a_1}{a_2} .
\end{aligned}
\]
\end{comment}
%%%%%%%%%%%%%%%%%%%%%%%%%%%%%%%%%%%%%%%%%%%%%%%%%%%%%%%%%%%%%%%%%%%%%%%%%%%%%%%

From this we compute:
\begin{align*}
Q_g(\tau_2(T_K),\tau_2(T_K))& = \eta_{S}(C_S(\pi_{13}(\tau_2(T_K))),C_S(\pi_{31}(\tau_2(T_K))))\\
& = \eta_{S}(b_1b_2-b_2b_2-b_1b_1, -a_2a_2-a_1a_2-a_1a_1) \\
& = -1 +2+2=3,\\
J_g(\tau_2(T_K),\tau_2(T_K)) & = \frac{1}{4}\nabla_g\Bigg(\tree{b_2}{b_1}{b_2}{b_1}, \tree{a_2}{a_1}{a_2}{a_1}\Bigg) =\frac{1}{4}\cdot 3=\frac{3}{4}.
\end{align*}

Then from equations \eqref{eq:bil-lindep} and \eqref{eq:lincomb-bil} we have that
\begin{equation}
\label{eq:computB1}
\frac{3}{4}r_1+3r_2=72.
\end{equation}

\textbf{For the figure eight knot $L$.}
The Conway polynomial and the Jones polynomial of $L$ are respectively (cf. \cite[Section 1.2]{ohtsuki3})
\[\nabla(z)=1-z^2 \quad \text{and} \quad V(L,t)=t^2-t+1-t^{-1}+t^{-2}.\]
From this we get $c_4(L)=0$ and $v_2(L)=6$, hence  by Propositions \ref{prop:trick} and \ref{prop:comput-lambdas} we have that 
 $\lambda(T_L)=\lambda(S^3_{L,1})=-1$ and 
\[
\lambda_2(T_L^2)-2\lambda_2(T_L)=\lambda_2(S^3_{L,1/2})+2\lambda_2(S^3_{L,1})=2(v_2(L)+\dfrac{5}{3}v_2^2(L)-60c_4(L))=132.
\]

Therefore by equation \eqref{eq:bil-calcul-d2} we get that
\begin{equation}
\label{eq:lincomb-bil2}
B_g(T_L,T_L) =96.
\end{equation}

Next we compute $\tau_2(T_L)$.
Consider the colored curves depicted in Figure~\ref{fig:eight-surgery}. These two curves form a symplectic basis of the first homology group of the subsurface bounded by $L$ that does not contain the marked disk.

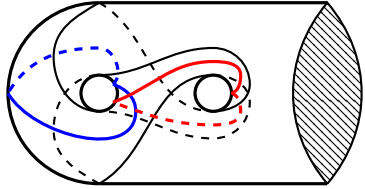
\begin{figure}[H]
	\begin{center}
\begin{tikzpicture}[scale=.6]
\draw[very thick] (-4.5,-2) -- (0.5,-2);
\draw[very thick] (-4.5,2) -- (0.5,2);
\draw[very thick] (-4.5,2) arc [radius=2, start angle=90, end angle=270];
\draw[very thick] (-4.5,0) circle [radius=.4];
\draw[very thick] (-2,0) circle [radius=.4];

\draw[thick] (-2,-0.4) to [out=0,in=-90] (-1.2,0.2)
to [out=90,in=0] (-2,1) to [out=180,in=0] (-4.5,0.4);
\draw[dashed,thick] (-2,0.4) to [out=0,in=90] (-1.2,-0.2)
to [out=-90,in=0] (-2,-1) to [out=180,in=0] (-4.5,-0.4);

\draw[ thick] (-4.5,-2) to [out=30,in=180] (-2,0.4);
\draw[dashed, thick] (-4.5,2) to [out=-30,in=180] (-2,-0.4);

\draw[ dashed,thick] (-4.5,-2) to [out=150,in=-90] (-5.5,-0.8) to [out=90,in=180] (-4.5,0.4);
\draw[ thick] (-4.5,2) to [out=-150,in=90] (-5.5,0.8) to [out=-90,in=180] (-4.5,-0.4);

\draw[dashed,very thick, blue] (-4.2,0.2) to [out=40,in=0] (-4.5,1) to [out=180,in=60] (-6.5,0);
\draw[very thick, blue] (-4.2,0.2) to [out=-20,in=90] (-3.7,-0.4) to [out=-90,in=-60] (-6.5,0);

\draw[dashed, very thick, red] (-1.6,0) to [out=-20,in=90] (-1.4,-0.4) to [out=-90,in=0] (-2,-0.7) to [out=180,in=-20] (-4.2,-0.2);
\draw[very thick, red] (-1.6,0) to [out=20,in=-90] (-1.4,0.4) to [out=90,in=0] (-2,0.7) to [out=180,in=20] (-4.2,-0.2);

\draw[thick,pattern=north west lines] (0.5,-2) to [out=130,in=-130] (0.5,2) to [out=-50,in=50] (0.5,-2);
\end{tikzpicture}
\end{center}
	\caption{Eight knot embedded in $\Sigma_{2,1}$}
	\label{fig:eight-surgery}
\end{figure}

Moreover the homology classes of blue and red curves are $a_1+b_1$ and $a_2+b_1-b_2$ respectively. Again by \cite[Lemma 2.3]{faes1} we have that

\[\tau_2(T_{L})=\frac{1}{2}\tree{a_1+b_1\quad\quad}{a_2+b_1-b_2\quad\quad}{\quad\quad a_1+b_1}{\quad\quad a_2+b_1-b_2}\]

The projections of $\tau_2(T_L)$ to $W(b^4)$, $W(ab^3)$, $W(a^4)$ and $W(a^3b)$ are:
\[
\begin{aligned}
\pi_{04}(\tau_2(T_L)) & = \frac{1}{2}\tree{b_1}{b_2}{b_1}{b_2}, \\
\pi_{40}(\tau_2(T_L))  & = \frac{1}{2}\tree{a_1}{a_2}{a_1}{a_2},
\end{aligned}
\qquad
\begin{aligned}
\pi_{13}(\tau_2(T_L)) & =-\tree{a_1}{b_1-b_2}{b_1}{b_2}-\tree{b_1}{a_2}{b_1}{b_2}, \\
\pi_{31}(\tau_2(T_L)) & = \tree{b_1}{a_2}{a_1}{a_2}+\tree{a_1}{b_1-b_2}{a_1}{a_2}.
\end{aligned}
\]

%%%%%%%%%%%%%%%%%%%%%%%%%%%%%%%%%%%%%%%%%%%%%%%%%%%%%%%%%%%%%%%%%%%%%%%%%%%%%%%
\begin{comment}
\begin{align*}
\pi_{13}(\tau_2(T_L)) & = \tree{a_1}{b_1-b_2}{b_1}{b_1-b_2}+\tree{b_1}{a_2}{b_1}{b_1-b_2} =-\tree{a_1}{b_1-b_2}{b_1}{b_2}-\tree{b_1}{a_2}{b_1}{b_2}, \\
\pi_{31}(\tau_2(T_L)) & = \tree{b_1}{a_2}{a_1}{a_2}+\tree{a_1}{b_1-b_2}{a_1}{a_2}.
\end{align*}
\end{comment}
%%%%%%%%%%%%%%%%%%%%%%%%%%%%%%%%%%%%%%%%%%%%%%%%%%%%%%%%%%%%%%%%%%%%%%%%%%%%%%%

From here we compute:
\begin{align*}
Q_g(\tau_2(T_L),\tau_2(T_L))& = \eta_{S}(C_S(\pi_{13}(\tau_2(T_L))),C_S(\pi_{31}(\tau_2(T_L))))\\
& =\eta_{S}(b_1b_2-b_2b_2+b_1b_1, -a_2a_2+a_1a_2+a_1a_1) \\
& = 1+2+2=5,\\
J_g(\tau_2(T_L),\tau_2(T_L)) & = \frac{1}{4}\nabla_g\Bigg(\tree{b_1}{b_2}{b_1}{b_2}, \tree{a_1}{a_2}{a_1}{a_2}\Bigg)= \frac{1}{4}\cdot 3=\frac{3}{4}.
\end{align*}

Then from equations \eqref{eq:bil-lindep} and \eqref{eq:lincomb-bil2} we have that
\begin{equation}
\label{eq:computB2}
\frac{3}{4}r_1+5r_2=96.
\end{equation}

Finally solving the system of equations \eqref{eq:computB1} and \eqref{eq:computB2} we get that $r_1=48$ and $r_2=12$.

\end{proof}
%%%%%%%%%%%%%%%%%%%%%%%%%%%%%%%%%%%%%%%%%%%%%%%%%%%%%%%%%%%%%%%%%%%

%The proof of Theorem~\ref{thm:filtisnotfull} can be used as a blueprint to further detect steps in the Johnson filtration of $\mathcal{S}_0$, depending on the following conjecture.
%
%%%%%%%%%%%%%%%%%%%%%%%%%%%%%%%%%%%%%%%%%%%%%%%%%%%%%%%%%%%%%
%\begin{conjecture}[Conj. 1.8 in~\cite{MSS}]
%%%%%%%%%%%%%%%%%%%%%%%%%%%%%%%%%%%%%%%%%%%%%%%%%%%%%%%%%%%%%
%For $k \geq 2$, the canonical map $\mathfrak{t}_{g,1}(k) \rightarrow \mathfrak{m}_{g,1}(k)$ is an isomorphism, equivalently, for any $k \geq 3$, the group $\mathcal{T}_{g,1}(k)$ is a finite index subgroup of the homomorphism induced by the Casson invariant $\lambda:\mathcal{M}_{g,1}(k) \rightarrow \mathbb{Q}$.
%\end{conjecture}
%%%%%%%%%%%%%%%%%%%%%%%%%%%%%%%%%%%%%%%%%%%%%%%%%%%%%%%%%%%%%
%
%Indeed, assuming the conjecture to be true, let $F: \mathcal{S}_0 \rightarrow \mathbb{Q}$ be an invariant of degree $m \geq 3$. Assume that for some $k \geq m$  the induced maps $F_g: \mathcal{M}_{g,1}(k) \rightarrow \mathbb{Q}$ are group homomorphisms. Then by the above conjecture, since $F$ vanishes on the subgroup $\mathcal{T}_{g,1}(2m+1)$, then $F_g:\mathcal{M}_{g,1}(2m+1) \rightarrow \mathbb{Q}$ is proportional to the Casson invariant, hence for some rational number $q$, the invariant $F-q\lambda$ vanishes on $\mathcal{M}_{g,1}(2m+1)$ and detects the difference between $\mathcal{S}_m$ and $\mathcal{S}_0$.

\end{document}